\documentclass[11pt, a4paper]{article}
\usepackage{amsmath,amssymb,amsthm}
\usepackage{mathtools}
\usepackage{xcolor}
\usepackage{caption}
\usepackage{subcaption}
\usepackage{tikz}
\usetikzlibrary{matrix,shapes,decorations.pathreplacing}
\usepackage{comment}
\usepackage{dsfont}
\usepackage{nicematrix}
\usepackage{enumitem}
\usepackage{kotex}
\usepackage{mathrsfs}
\usepackage{tabularx}
\usepackage{hyperref}
\usepackage[affil-it]{authblk} % 저자 및 소속 정렬용
\hypersetup{colorlinks,breaklinks,
             linkcolor=black,urlcolor=blue,
             anchorcolor=blue,citecolor=blue}
\usepackage[nameinlink, capitalise, noabbrev]{cleveref}
\usepackage[table]{xcolor}
\usepackage{overpic}
\usepackage{geometry}
\newtheorem{theorem}{Theorem}[section]

\newtheorem{lemma}[theorem]{Lemma}
\newtheorem{proposition}[theorem]{Proposition}

\newtheorem{remark}[theorem]{Remark}

\DeclareMathOperator*{\argmax}{arg\,max}

\title{\Large \textbf{Stability of Phase-Locked States in Signed Kuramoto Networks: Structure versus Adaptation}}

\author[1]{Jaeyoung Yoon} 
\author[1,2]{Christian Kuehn} 

\affil[1]{Department of Mathematics, School of Computation, Information and Technology, Technical University of Munich, Garching bei M\"unchen 85748, Germany} 
\affil[2]{Complexity Science Hub Vienna, Vienna, Austria}
\date{} % 또는 특정 날짜 입력

\begin{document}

\maketitle

% \keywords{Adaptive Kuramoto model, Phase-locked states, Signed network, Spectral stability, Network adaptation}

\begin{abstract}
Adaptive Kuramoto models admit a variety of nontrivial phase-locked configurations, including antipodal and rotating-wave states. A central open question is whether the observed persistence of such configurations can be attributed to intrinsic properties of the associated signed interaction networks, or whether it relies essentially on adaptive coupling dynamics. To address this question, we study the stability of antipodal and rotating-wave phase configurations on fixed signed networks that preserve the same phase symmetries but are not generated by adaptive dynamics. We show that for two canonical classes of static signed networks, stability is highly constrained, with unstable modes persisting under parameter variations generically, and we characterize how adaptive coupling influences invariant sets and basins of attraction for the configurations where stability is permitted. Taken together, these results show that while static network structure imposes severe constraints on the stability of phase-locked configurations, adaptive coupling dynamics organize and delineate their robustness when stability is permitted.
\end{abstract}

\vspace{0.5cm}
\noindent \textbf{Keywords:} Adaptive Kuramoto model, Phase-locked states, Signed network, Spectral stability, Network adaptation

\section{Introduction}
Classical models of collective dynamics typically assume fixed interaction structures. However, in many real systems the interaction network coevolves with the system’s state, leading to adaptive network dynamics in which structure and dynamics are inseparably coupled \cite{GB2008,Sawicki2023}. Such adaptive network dynamical systems are characterized by a two-way coupling \cite{BGKKY2023}: the node dynamics depend on the network, while the network adapts according to the node states \cite{KYSN2017}. This coevolution can profoundly alter collective behavior and has been shown to generate forms of organization that are not supported by static networks alone.

Networks of phase oscillators, and in particular Kuramoto-type models, provide a canonical setting for studying collective dynamics on adaptive networks. Beyond complete synchronization, these systems admit a variety of nontrivial phase-locked states, such as clustered, antipodal, and rotating-wave configurations \cite{AS2004,Jaros2018,Haugland2021}. Recent work has demonstrated that adaptive coupling rules can robustly generate such states, often accompanied by the emergence of structured, signed interaction networks \cite{SYT2002,RZ2007}.

Despite these advances, an important conceptual question remains unresolved. When a phase-locked state emerges in an adaptive system, it is typically accompanied by the formation of a structured interaction network, often with signed couplings reflecting excitation and inhibition \cite{SAB2019,Gallier2016,HS2011,Altafini2013,Cartwright1956}. This raises a natural question: to what extent is the stability of the observed phase-locked state a property of the induced network structure itself, and to what extent does it rely on the ongoing adaptive dynamics? In most existing analyses, these two aspects are inseparable, as stability is examined within the full adaptive system where phases and couplings coevolve \cite{BSY2019,HNP2016}.

In this work, we address this question by separating the roles of network structure and adaptive dynamics in stability. We do so by fixing the interaction network \cite{DJD2019,CMM2018} and analyzing the stability of phase-locked states on the resulting static signed network. This allows us to identify structural constraints that are difficult to disentangle in fully adaptive systems. Our results are consistent with existing stability analyses of adaptive models \cite{BSY2019} and provide a complementary perspective on the role played by static interaction structure.

We focus on two canonical classes of signed networks that arise naturally as asymptotic interaction patterns in adaptive Kuramoto dynamics. The first class consists of block-structured signed networks, which capture clustered patterns of excitation and inhibition associated with synchronized and antipodal phase configurations. The second class comprises locally excitatory–globally inhibitory (LEGI) networks, characterized by short-range positive interactions and long-range negative couplings, a structure closely linked to rotating-wave solutions \cite{WSG2006}. These networks preserve the symmetries of the corresponding phase configurations and therefore provide representative settings for probing structural stability.

Our analysis reveals that static signed networks impose severe constraints on the stability of nontrivial phase-locked states. In block-structured networks, antipodal configurations are generically unstable, with stability occurring only in highly restricted and exceptional cases. Similarly, in LEGI networks, rotating-wave states are supported only within narrow parameter regimes. In both cases, instabilities persist even when the static network fully respects the symmetry of the phase configuration, demonstrating that symmetry alone is not sufficient for stability.

We then return to the adaptive Kuramoto model to clarify the dynamical role of coupling adaptation, whose time scale is controlled by a parameter $\varepsilon$. Rather than directly stabilizing all unstable modes of the corresponding static system, adaptive dynamics organize invariant sets and basins of attraction in phase space \cite{Menck2013}. This mechanism explains how phase-locked states that are structurally fragile on static networks can nevertheless appear robust in adaptive systems. We derive explicit sufficient conditions under which complete synchronization or antipodal synchronization emerges and quantify how the adaptation rate controls the relaxation of initial conditions and the size of admissible basins of attraction.

Taken together, our results show that the robustness of nontrivial phase-locked states in adaptive oscillator networks cannot, in general, be inferred from static interaction structure alone. Instead, robustness emerges from the ongoing coevolution between dynamics and coupling architecture. More broadly, this work highlights the importance of disentangling structural and dynamical effects when analyzing stability in adaptive complex systems and provides a framework for systematically assessing how adaptive mechanisms shape collective behavior.

\vspace{.2cm}

The remainder of the paper is organized as follows. In \Cref{Sec_Pre}, we review relevant results on phase-locked solutions of the adaptive Kuramoto model and introduce canonical classes of signed networks motivated by these solutions. \Cref{Sec_stability} is devoted to the stability analysis of phase-locked equilibria on fixed signed networks, where we identify intrinsic structural constraints in block-structured and locally excitatory–globally inhibitory networks. In \Cref{sec_special}, we return to the adaptive Kuramoto model and derive explicit sufficient conditions for complete and antipodal synchronization, quantifying how coupling adaptation organizes invariant sets and basins of attraction. Finally, \Cref{sec_conclusion} summarizes the main findings and discusses their implications for stability and robustness in adaptive complex networks.
%%%%%%%%%%%%%%%%%%%%%%%%%%%%%%%%%%%%%%%%%%%%%%%%%%%%%%%%%%%%%%%%%%%%%%%%%%%%%%%%%%%%%%%%%%%%%%%

\section{Preliminaries}\label{Sec_Pre}
In this section, we outline the framework underlying our analysis. \Cref{subsec_review} reviews relevant results on phase-locked solutions of the adaptive Kuramoto model from \cite{BSY2019}. Rather than providing a complete survey of all possible cluster states, we restrict attention to antipodal and rotating-wave configurations and to the corresponding interaction patterns discussed in the asymptotic analysis of \cite{BSY2019}. In \Cref{subsec_network}, we select representative static signed networks motivated by these configurations, which will be used in the subsequent sections to examine stability properties and the effects of adaptive coupling dynamics.

\subsection{Previous results}\label{subsec_review}
We briefly recall relevant results from \cite{BSY2019} concerning phase-locked solutions of the adaptive Kuramoto model,
\begin{align}\label{adapkm}
    \begin{aligned}
        \begin{dcases}
            \frac{d\theta_i}{dt}&=\omega-\frac{1}{N}\sum_{j=1}^N\kappa_{ij}\sin(\theta_i-\theta_j+\alpha),\\
            \frac{d\kappa_{ij}}{dt}&=-\varepsilon(\sin(\theta_i-\theta_j+\beta)+\kappa_{ij}),\quad\forall~i,j\in[N],
        \end{dcases}
    \end{aligned}
\end{align}
as a reference point for the stability analysis carried out in this paper. In \cite{BSY2019}, the \emph{$n$-th order parameters} $(Z_n,R_n,\Psi_n)$ associated with a phase configuration
$\Theta=(\theta_1,\ldots,\theta_N)\in\mathbb T^N$ are defined by
\begin{align*}
    Z_n(\Theta)=R_n(\Theta)e^{i\Psi_n(\Theta)}:=\frac{1}{N}\sum_{j=1}^Ne^{in\theta_j},\quad\forall~n\in\mathbb N,
\end{align*}
where $R_n\ge 0$ and $\Psi_n\in\mathbb T^1$.
Based on the second-order parameter $R_2(\Theta)$, phase-locked states can be classified into three qualitatively
distinct types:
\begin{itemize}[itemsep=0.2em]
\item splay-type clusters, characterized by $R_2(\Theta)=0$,
\item antipodal clusters, characterized by $R_2(\Theta)=1$, i.e., $\theta_i\in\{0,\pi\}$ for all $i\in[N]$,
\item double antipodal-type clusters, where $\theta_i\in\{0,\pi,\psi,\psi+\pi\}$ for all $i\in[N]$ with some $\psi\in(0,\pi)$.
\end{itemize}

The main classification result of \cite{BSY2019} can be summarized as follows.

\begin{proposition}[Phase-locked solutions]\label{BSY-1}
The system \eqref{adapkm} admits phase-locked solutions of the form
\[\theta_i(t)=\Omega t+\phi_i,\quad\kappa_{ij}=-\sin(\phi_i-\phi_j+\beta),\quad\forall~i,j\in[N],\]
where the phase offsets $\phi_i$ are constant in time, if and only if the configuration
$\{\phi_i\}_{i\in[N]}$ belongs to one of the three classes listed above.
More precisely,
\begin{enumerate}[label=(\roman*),itemsep=0.2em]
\item $\{\phi_i\}$ forms a splay-type cluster,
\item $\{\phi_i\}$ forms an antipodal cluster,
\item $\{\phi_i\}$ forms a double antipodal-type cluster, where
      $\phi_i\in\{0,\pi,\psi_m,\psi_m+\pi\}$ for some $m\in\{1,\ldots,N-1\}$,
      and $\psi_m$ is the unique solution (modulo $2\pi$) of
      \[
      \frac{N-m}{m}\sin(\psi-\alpha-\beta)=\sin(\psi+\alpha+\beta),
      \]
      with exactly $m$ indices satisfying $\phi_i\in\{0,\pi\}$.
\end{enumerate}
\end{proposition}
\Cref{BSY-1} characterizes phase-locked solutions of the full adaptive system \eqref{adapkm} and identifies the coupling structure realized along such solutions. In particular, it does not address whether the same phase configurations can be stable when the coupling network is fixed. The purpose of the present work is to address this question by isolating structural instabilities of static signed networks. In the next subsection, we fix representative signed network structures suggested by these phase configurations and analyze their spectral stability.
%%%%%%%%%%%%%%%%%%%%%%%%%%%%%%%%%%%%%%%%%%%%%%%%%%%%%%%%%%%%%%%%%%%%%%%%%%%%%%%%%%%%%%%%%%%%%%%
\subsection{Canonical network structures induced by phase-locked states}\label{subsec_network}
We now examine the static interaction networks associated with the phase-locked solutions described in \Cref{BSY-1}. Rather than providing an exhaustive account of all possible configurations, we focus on identifying canonical network structures that will serve as representative cases for the stability analysis in \Cref{Sec_stability}.

\paragraph{Block-structured networks.}
If the phase configuration $\{\phi_i\}$ consists of one or two antipodal phase classes, the corresponding coupling matrix $\kappa_{ij}=-\sin(\phi_i-\phi_j+\beta)$ has a block-constant structure. In particular, complete synchronization yields a uniformly weighted complete graph, whereas antipodal configurations induce a signed two-block structure (See \Cref{fig_mat}).
\begin{figure}[h]
    \centering
    \begin{subfigure}[b]{0.45\textwidth}
        \centering
        \begin{tikzpicture}
            \node at (-0.6,0.8) {$K=$};
            
            \draw[thick] (0,2) -- (0,-0.5);
            \draw[thick] (0,2) -- (0.2,2);
            \draw[thick] (0,-0.5) -- (0.2,-0.5);
        
            \draw[thick] (2.7,2) -- (2.7,-0.5);
            \draw[thick] (2.7,2) -- (2.5,2);
            \draw[thick] (2.7,-0.5) -- (2.5,-0.5);
            
            \fill[pink!40] (0.1,1.9) rectangle (2.6,-0.4);
            \node at (1.35,0.8) {$-\sin\beta$};
        \end{tikzpicture}
        \caption{complete synchronization}
        \label{subfig_comp}
    \end{subfigure}
    \begin{subfigure}[b]{0.45\textwidth}
        \centering
        \begin{tikzpicture}
            \node at (-0.6,0.8) {$K=$};
            
            \draw[thick] (0,2) -- (0,-0.5);
            \draw[thick] (0,2) -- (0.2,2);
            \draw[thick] (0,-0.5) -- (0.2,-0.5);
        
            \draw[thick] (2.7,2) -- (2.7,-0.5);
            \draw[thick] (2.7,2) -- (2.5,2);
            \draw[thick] (2.7,-0.5) -- (2.5,-0.5);

            \draw (0.1,0.6) -- (2.6,0.6);
            \draw (1.5,-0.4) -- (1.5,2);

            \fill[pink!40] (0.1,0.7) rectangle (1.4,1.9);
            \fill[pink!40] (1.6,-0.4) rectangle (2.6,0.5);
            \draw[pink, thick] (0.1,-0.4) rectangle (1.4,0.5);
            \draw[pink, thick] (1.6,0.7) rectangle (2.6,1.9);

            \node at (2.1,1.25) {$\sin\beta$};
            \node at (0.75,1.25) {$-\sin\beta$};
            \node at (0.75,0) {$\sin\beta$};
            \node at (2.1,0) {$-\sin\beta$};
        \end{tikzpicture}
        \caption{antipodal-states}
        \label{subfig_anti}
    \end{subfigure}
    \caption{Graph structures $K=(\kappa_{ij})$ when equilibrium are in the form of (a) complete synchronization and (b) antipodal states.}
    \label{fig_mat}
\end{figure}
These networks constitute natural examples of block-structured signed graphs and will be analyzed in \Cref{subsec_block}.

\paragraph{Locally excitatory--globally inhibitory networks.}
Splay-type and rotating-wave phase configurations generate interaction networks characterized by local excitation and long-range inhibition.
A prototypical example is the circular band network, in which each node excites a fixed number of nearest neighbors and inhibits all others.
The corresponding adjacency matrix exhibits a characteristic banded pattern with positive near-diagonal entries and negative long-range couplings (see \Cref{fig_beta}).
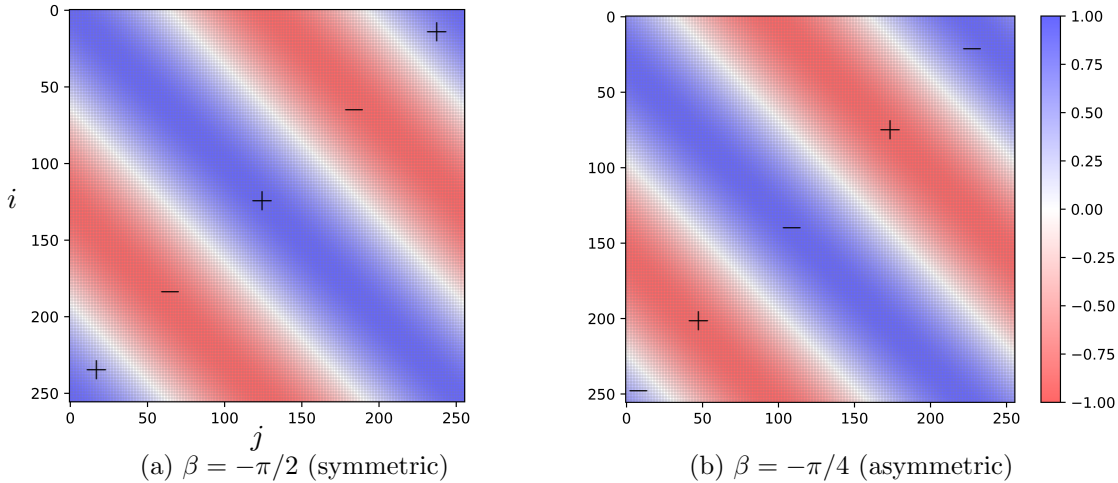
\begin{figure}[h]
    \centering
    
    \begin{subfigure}[b]{0.45\textwidth}
        \begin{overpic}[width=0.84\textwidth]{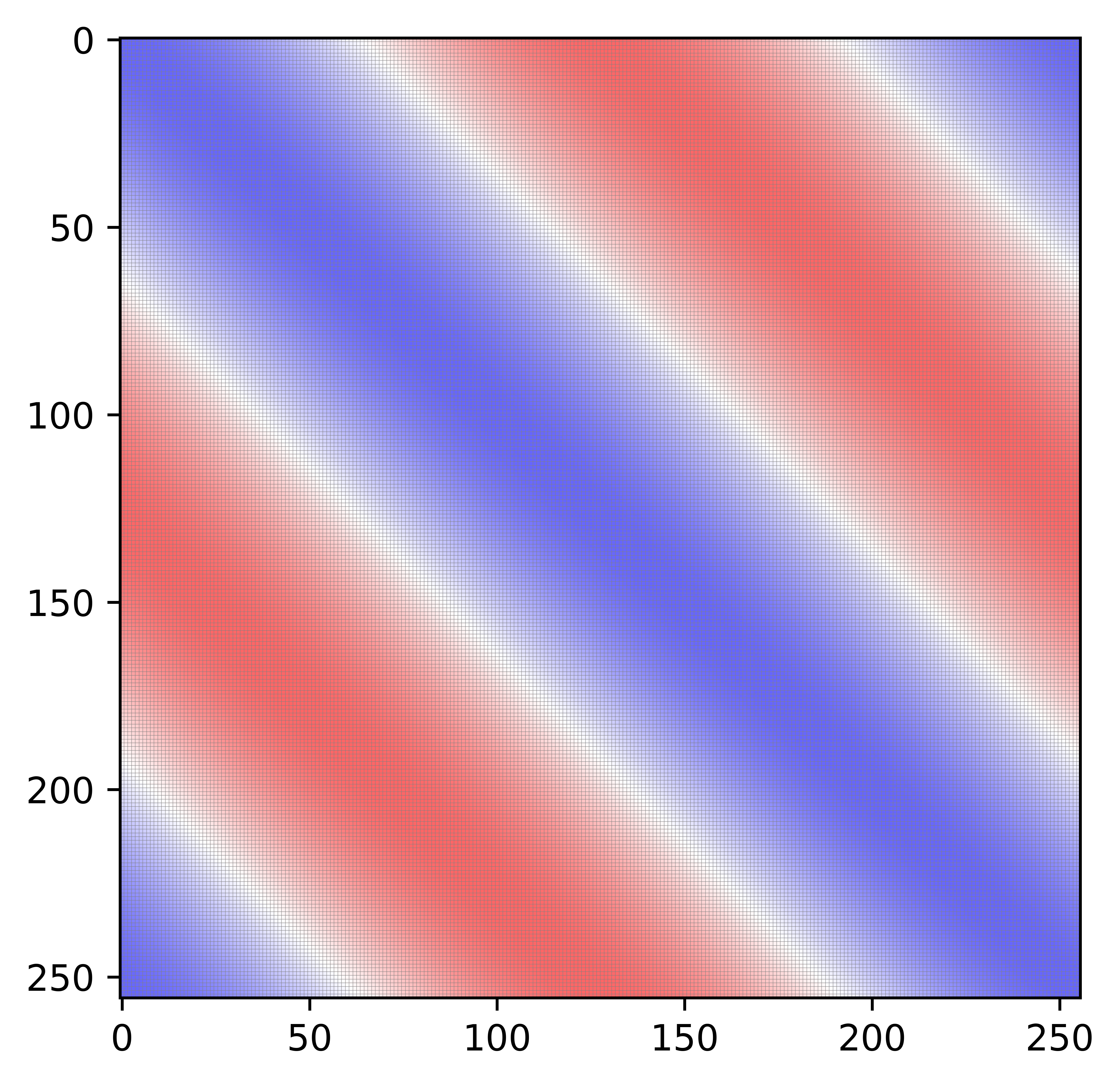}
            \put(30,30){$-$}
            \put(70,70){$-$}
            \put(50,50){$+$}
            \put(88,87){$+$}
            \put(14,13){$+$}
            \put(-3,50){$i$}
            \put(50,-2){$j$}
        \end{overpic}
        \caption{$\beta=-\pi/2$ (symmetric)}
        \label{subfig_beta2}
    \end{subfigure}
    \begin{subfigure}[b]{0.45\textwidth}
        \begin{overpic}[width=\textwidth]{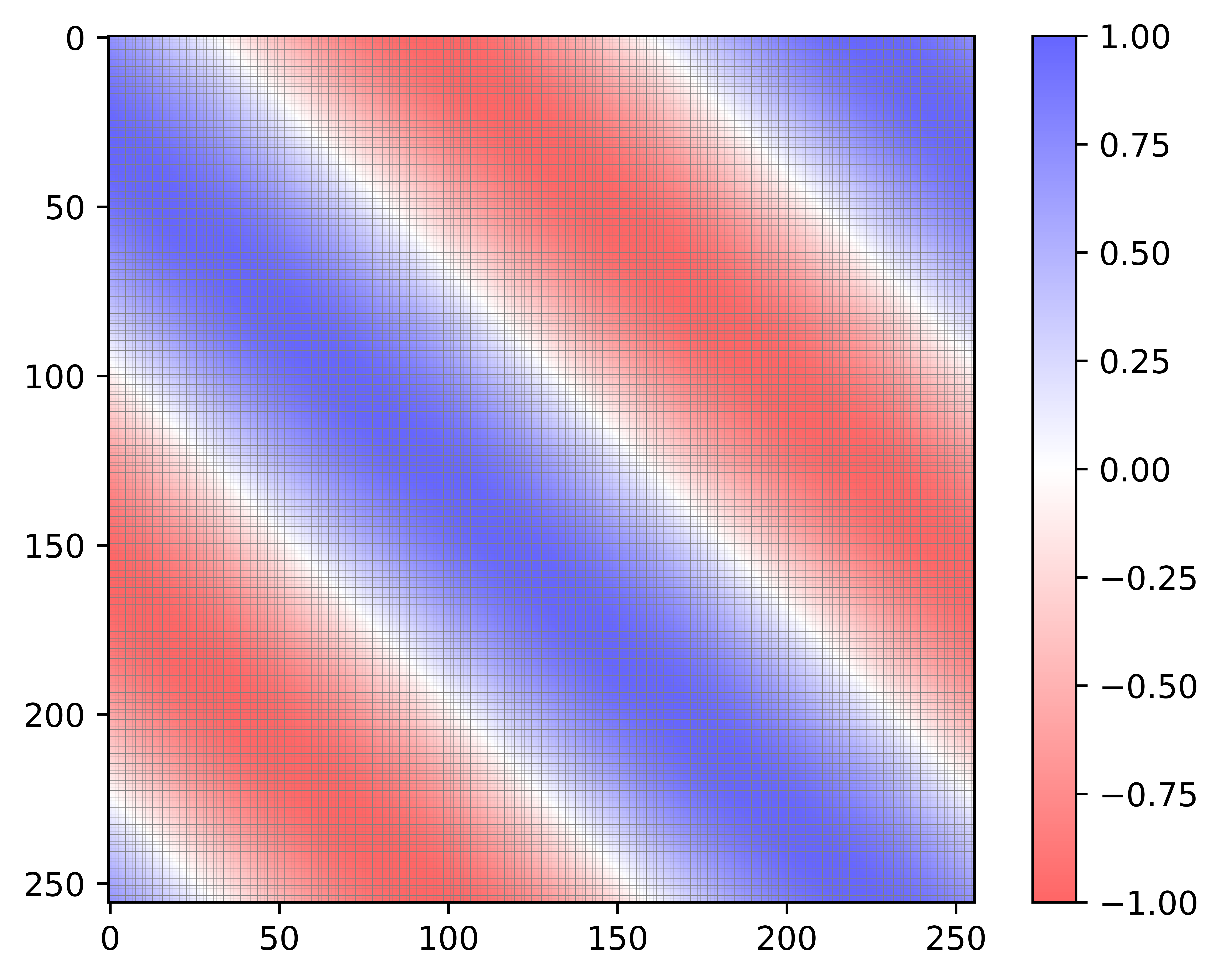}
            \put(9,7){$-$}
            \put(20,20){$+$}
            \put(37,37){$-$}
            \put(55,55){$+$}
            \put(70,70){$-$}
        \end{overpic}
        \caption{$\beta=-\pi/4$ (asymmetric)}
        \label{subfig_beta4}
    \end{subfigure}
    \caption{Color-coded adjacency matrix of network $\kappa_{ij}=-\sin(\theta_i-\theta_j+\beta)$}
    \label{fig_beta}
\end{figure}
This locally excitatory--globally inhibitory (LEGI) structure captures the essential features of networks induced by rotating-wave equilibria and will be analyzed in \Cref{subsec_legi}.

\vspace{.5cm}

In summary, the phase-locked solutions of the adaptive Kuramoto model naturally motivate two canonical
classes of static signed networks: block-structured networks and locally excitatory--globally inhibitory
networks.
The local stability of equilibria on these networks is investigated in the next section.
%%%%%%%%%%%%%%%%%%%%%%%%%%%%%%%%%%%%%%%%%%%%%%%%%%%%%%%%%%%%%%%%%%%%%%%%%%%%%%%%%%%%%%%%%%%%%%%
\section{Stability of equilibria under fixed networks}\label{Sec_stability}
In this section, we investigate the local stability of phase-locked equilibria arising on fixed networks. Motivated by the network structures obtained in \Cref{Sec_Pre}, we focus on two canonical classes of signed graphs that serve as representative fixed networks for the stability analysis of adaptive Kuramoto dynamics. For each class, we analyze the linearized dynamics around representative phase-locked states and characterize stability in terms of the spectral properties of the associated Jacobian matrices.

Specifically, \Cref{subsec_block} is devoted to block-structured signed networks, which capture clustered interaction patterns underlying synchronized and antipodal configurations, while \Cref{subsec_legi} examines locally excitatory--globally inhibitory (LEGI) networks, where spatially structured competition between excitation and inhibition gives rise to rotating-wave solutions.

%%%%%%%%%%%%%%%%%%%%%%%%%%%%%%%%%%%%%%%%%%%%%%%%%%%%%%%%%%%%%%%%%%%%%%%%%%%%%%%%%%%%%%%%%%%%%%%
\subsection{Kuramoto dynamics on block-structured signed networks}\label{subsec_block}
In this subsection, we study the local stability of phase-locked equilibria on block-structured signed networks. To isolate the role of network structure from adaptive dynamics, we consider the non-adaptive limit of system \eqref{adapkm} by setting $\epsilon = 0$, which yields the following dynamics on a fixed network $K$:
\begin{align*}
    \frac{d\theta_i}{dt}=\omega-\frac{1}{N}\sum_{j=1}^N\kappa_{ij}\sin(\theta_i-\theta_j+\alpha),
\end{align*}
where the coupling strengths $\kappa_{ij}$ are constant in time. Such networks arise naturally in adaptive Kuramoto dynamics, serving as asymptotic interaction structures associated with antipodal and multi-cluster states \cite{BSY2019}. Our goal is to identify how the network parameters and group sizes influence the stability of these canonical equilibria, particularly complete and antipodal synchronization.

We consider a block-structured network in which the vertex set $[N]$ is partitioned into $M$ disjoint groups,
\[
\bigcup_{m=1}^M \mathcal G_m = [N],
\quad
\mathcal G_m \cap \mathcal G_\ell = \emptyset, \quad\mbox{for}~~m\ne l.
\]
The coupling strength between two oscillators depends only on whether they belong to the same group or to different groups, i.e.,
interactions within a group have strength $a \in \mathbb R$, while interactions between distinct groups have strength $b \in \mathbb R$. Accordingly, the adjacency matrix $K$ takes the block form
\begin{align}\label{bl_st}
    K=
    \begin{bmatrix}
        \colorbox{pink!40}{ $A_1$ } & \fcolorbox{blue}{white}{$B_{12}$} & \cdots & \fcolorbox{blue}{white}{$B_{1M}$}\\
        \fcolorbox{blue}{white}{$B_{21}$} & \colorbox{pink!40}{ $A_2$ } & \cdots & \fcolorbox{blue}{white}{$B_{2M}$}\\
        \vdots & \vdots & \ddots & \vdots\\
        \fcolorbox{blue}{white}{$B_{M1}$} & \fcolorbox{blue}{white}{$B_{M2}$} & \cdots & \colorbox{pink!40}{ $A_M$ }
    \end{bmatrix},
\end{align}
where $A_m = a\,\mathbf 1_{|\mathcal G_m|\times|\mathcal G_m|}$ and
$B_{m\ell} = b\,\mathbf 1_{|\mathcal G_m|\times|\mathcal G_\ell|}$.
Here, $\mathbf 1_{p\times q}$ denotes the $p\times q$ matrix with all entries equal to one. When the intra-group interaction is positive and the inter-group interaction is negative, i.e., $a>0>b$, the resulting signed network is commonly referred to as a \emph{weakly structurally balanced network} \cite{SAB2019,ZC2014}.

The block structure \eqref{bl_st} of the network suggests that phase-locked equilibria may arise in which oscillators within each group share a common phase. Motivated by this observation, we consider equilibria of the reduced form
\[\theta_i = c_m\quad i\in\mathcal G_m,~~m\in[M].\]
Substituting this ansatz into the Kuramoto dynamics yields the reduced equilibrium conditions
\begin{equation}\label{eq:reduced_equilibrium}
    0 = \frac{b}{N}\sum_{\tilde m=1}^M |\mathcal G_{\tilde m}|\sin\big(c_{\tilde m}-c_m\big),\quad\forall~ m\in[M].
\end{equation}
To reformulate the state \eqref{eq:reduced_equilibrium}, it is convenient to introduce the weighted order parameter
$(R,\Psi)\in\mathbb R_{\ge0}\times\mathbb T^1$ defined by
\[
R(\Theta)e^{i\Psi(\Theta)}=\frac{1}{N}\sum_{j=1}^N e^{i\theta_j}.
\]
Under the group-wise consensus ansatz \eqref{eq:reduced_equilibrium}, this reduces to
\[R(c)e^{i\Psi(c)} :=\frac{1}{N}\sum_{m=1}^M |\mathcal G_m|e^{ic_m},\]
one readily verifies that \eqref{eq:reduced_equilibrium} implies either $R(c)=0$ or
\[\Psi(c)-c_m \in \{0,\pi\},\quad\forall~m\in[M].\]

The case $R(c)=0$ corresponds to a splay-type configuration with $M-2$ degrees of freedom.
For $M\ge3$, this yields a continuous family of equilibria, and standard linearization techniques are therefore not directly applicable.
We exclude this case from the present analysis.
When $M=2$, the condition $R(c)=0$ reduces to an antipodal configuration.
The case $M=1$ trivially yields complete synchronization.

Consequently, within the block-structured network \eqref{bl_st}, two canonical phase-locked equilibria naturally arise, that is, the completely synchronized state and the antipodal state. In the following, we analyze their local stability by examining the spectra of the corresponding Jacobian matrices.

\paragraph{Complete synchronization}
In this case, the Jacobian matrix is a multiple of Laplacian matrix of the network $K=(\kappa_{ij})$. In particular, it takes the form
\begin{align*}
    J(\Theta^\infty)=-\frac{1}{N}(D_K-K),
\end{align*}
where the degree matrix $D_K$ is defined by $(D_K)_{ii}=\sum_{j=1}^N\kappa_{ij}$. In the following proposition, we characterize the local stability of an antipodal equilibrium by computing the eigenvalues of the matrix $D_K-K$.
\begin{proposition}[Local Stability of Complete Synchronization]\label{prop_eigval_com_syn}
    The Laplacian matrix $L = D_K - K$ of a block-structured graph $K$ as in \eqref{bl_st} has the following eigenvalues:
    \begin{align*}
        &\lambda = a|\mathcal G_m| + b\bigl(N - |\mathcal G_m|\bigr),
        \quad \text{with multiplicity } |\mathcal G_m| - 1,
        \quad m = 1,\ldots,M,\\
        &\lambda = 0 \quad \text{(simple)}, \qquad
        \lambda = bN \quad \text{with multiplicity } M - 1.
    \end{align*}
\end{proposition}
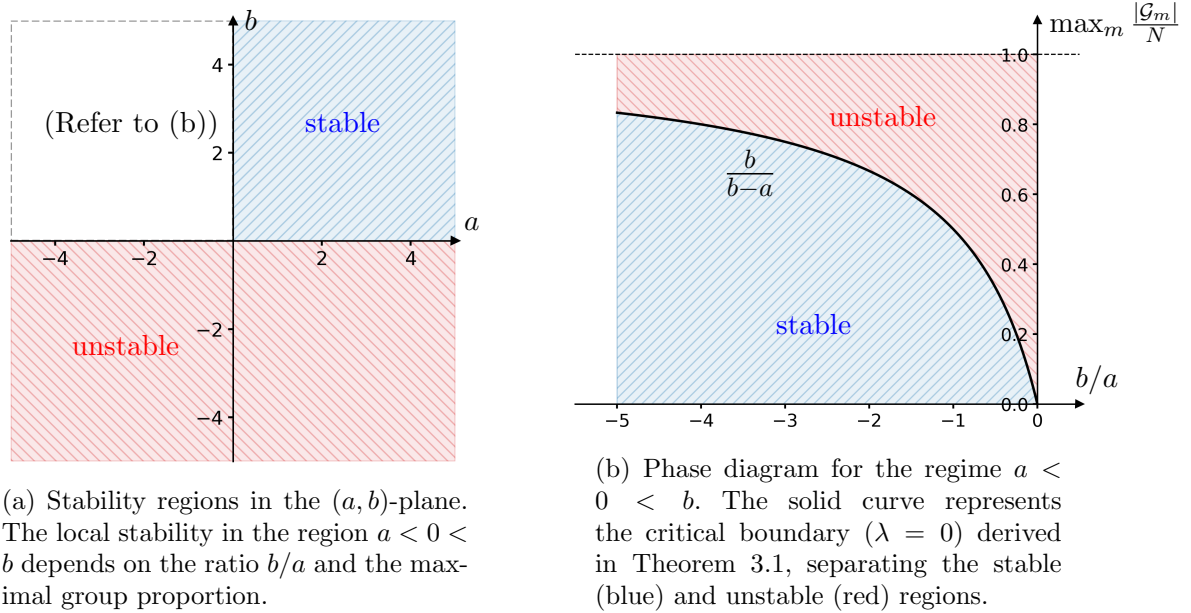
\begin{figure}[t]
    \centering
    \captionsetup[subfigure]{width=0.8\linewidth}
    \begin{subfigure}[b]{0.48\textwidth}
    \centering
        \begin{overpic}[width=0.8\textwidth]{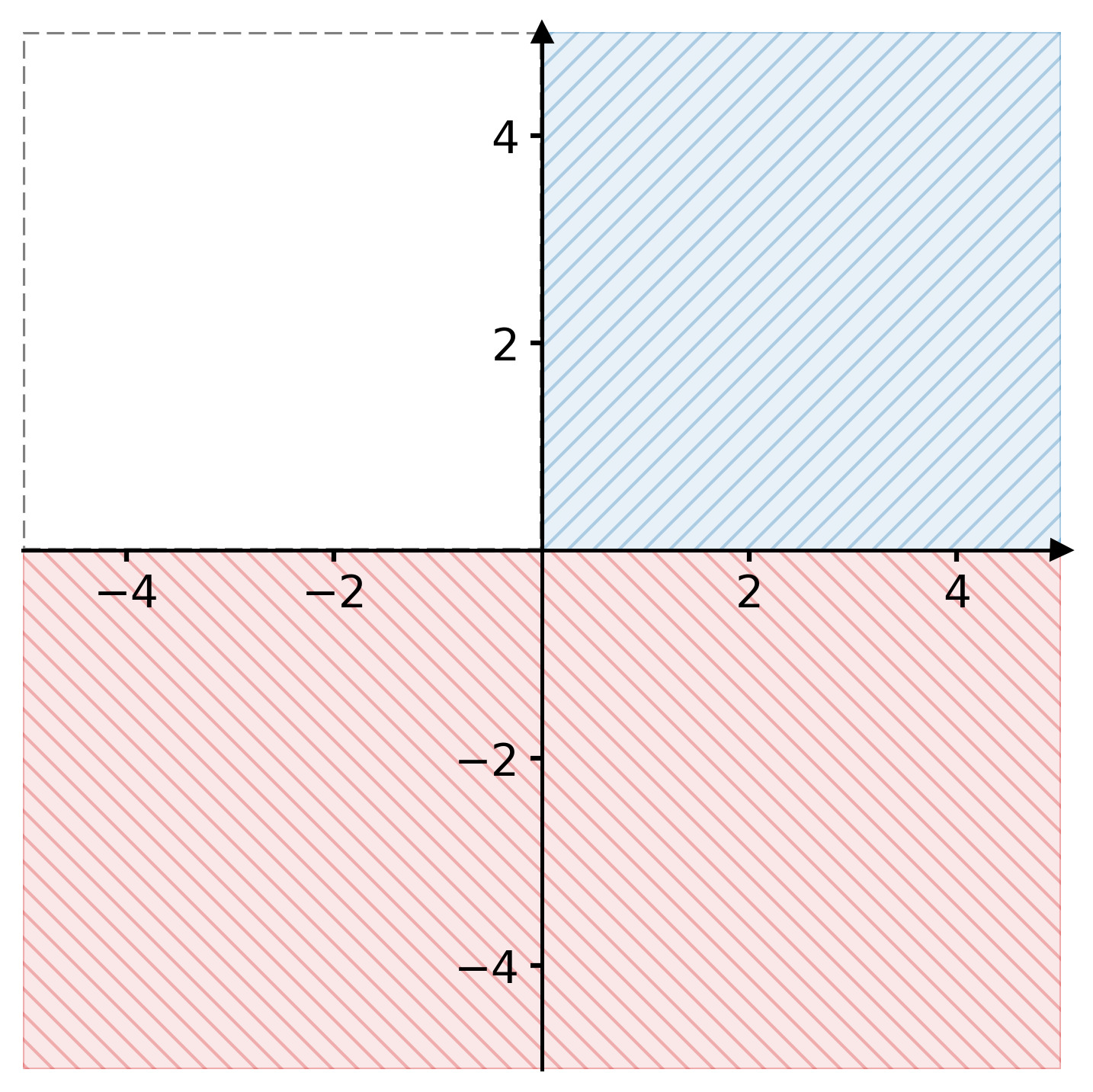}
            \put(99, 52){$a$}
            \put(52, 95){$b$}
            \put(65, 73){\color{blue} stable}
            \put(15, 25){\color{red} unstable}
            \put(9, 73){(Refer to (b))}
        \end{overpic}
        \caption{Stability regions in the $(a, b)$-plane. The local stability in the region $a<0<b$ depends on the ratio $b/a$ and the maximal group proportion.}
    \end{subfigure}
    \begin{subfigure}[b]{0.48\textwidth}
    \centering
        \begin{overpic}[width=0.9\textwidth]{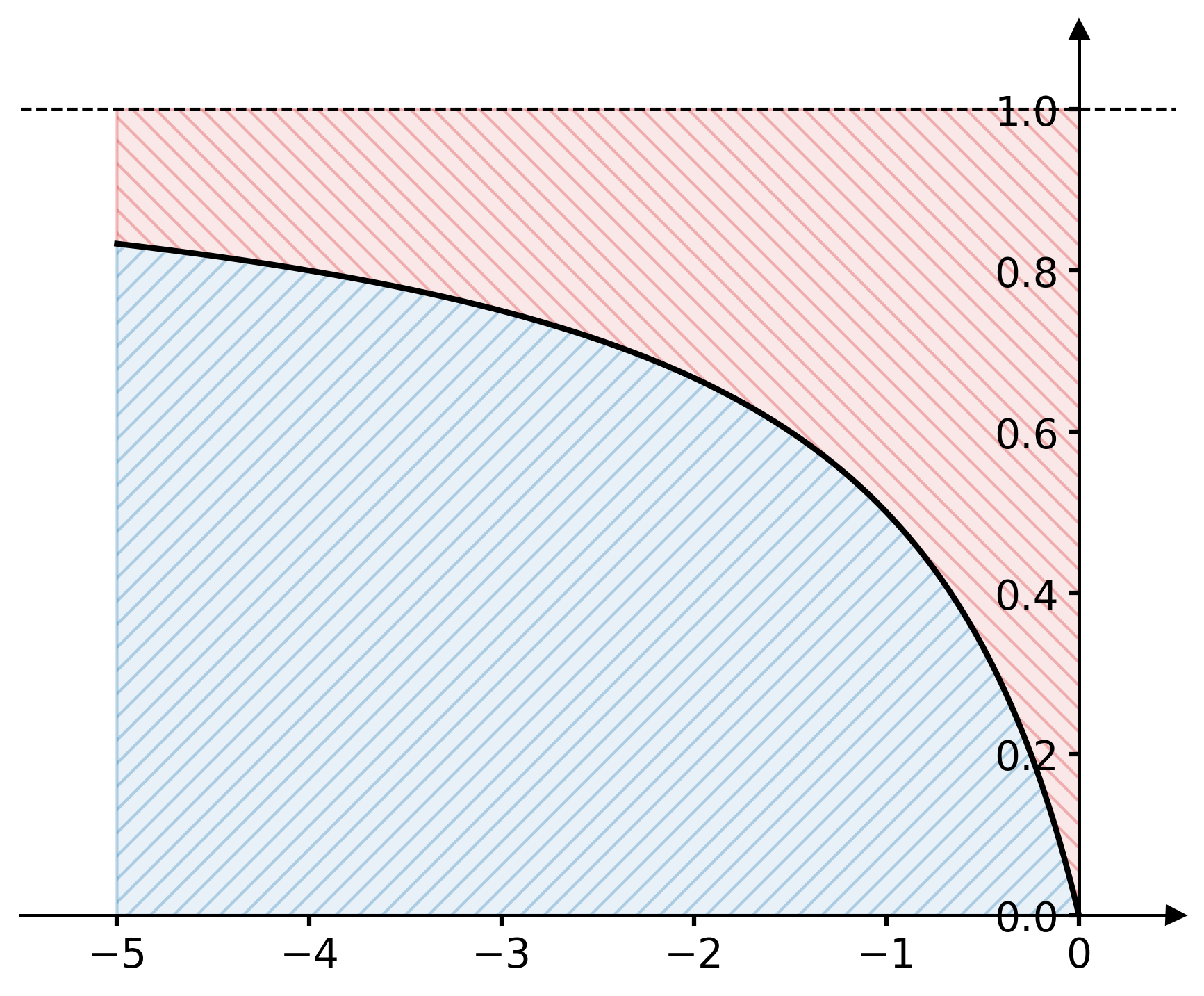}
            \put(97, 10){$b/a$}
            \put(92, 78){$\max_m \frac{|\mathcal{G}_m|}{N}$}
            \put(40, 20){\color{blue} stable}
            \put(50, 60){\color{red} unstable}
            \put(30, 49){\Large$\frac{b}{b-a}$}
        \end{overpic}
        \caption{Phase diagram for the regime $a<0<b$. The solid curve represents the critical boundary ($\lambda=0$) derived in \Cref{prop_eigval_com_syn}, separating the stable (blue) and unstable (red) regions.}
    \end{subfigure}
    \caption{Local stability of the complete synchronization state based on \Cref{prop_eigval_com_syn}. The region where the stability cannot be determined just by $a$ and $b$, which was left blank in (a), is depicted in (b) based on $b/a$ and the maximal group proportion.}
    \label{figure_stability}
\end{figure}
We omit the proof, as it is analogous to the proof of \Cref{prop_eigval_antipodal}, which is provided in the Appendix. \Cref{prop_eigval_com_syn} characterizes the spectrum of the Laplacian for the block-structured network and hence yields an explicit stability criterion for complete synchronization in terms of $(a,b)$ and the group sizes. In particular, the complete synchronized equilibrium is unstable when $b<0$ with $M\ge2$. In the special case $M=1$, the system reduces to the classical Kuramoto model of identical oscillators on a complete network with $a>0$, in which the fully synchronized state is stable and this is a well-known fact. \Cref{figure_stability} summarizes the resulting stability region.

\paragraph{Antipodal states.}
We consider an \emph{antipodal equilibrium}, meaning that the oscillators are partitioned into two nontrivial phase classes such that the phase difference between the two classes is exactly $\pi$.
Since the dynamics are invariant under global phase shifts, the common phase of each class can be chosen arbitrarily.
Without loss of generality, we fix a representative antipodal configuration by setting
\[
\mathcal G^{(0)} \cup \mathcal G^{(\pi)} = [M],
\qquad
\mathcal G^{(0)} \cap \mathcal G^{(\pi)} = \emptyset,
\]
and defining
\begin{align}\label{anti_cm}
c_m=
\begin{dcases}
0, & m\in\mathcal G^{(0)},\\
\pi, & m\in\mathcal G^{(\pi)}.
\end{dcases}
\end{align}
Let $\Theta^\infty$ be an antipodal equilibrium satisfying \eqref{anti_cm}.
Then, for $i\in\mathcal G_m$ and $j\in\mathcal G_\ell$, the phase differences satisfy
\[
\theta_j^\infty-\theta_i^\infty = c_\ell - c_m \in \{0,\pi\},
\]
and hence $\cos(c_m-c_\ell)\in\{1,-1\}$.
As a result, the Jacobian can be written in the Laplacian form
\[
J(\Theta^\infty) = -\frac{1}{N}\big(D-A\big),
\]
where the matrix $A$ has a block-constant structure with entries given by
\begin{align}\label{A_def}
\begin{aligned}
A_{ij}=
\begin{dcases}
\kappa_{ii}, & i=j,\\[0.2em]
\kappa_{ij}\cos(c_m-c_\ell), & i\neq j,
\end{dcases}
\qquad
i\in\mathcal G_m,\ j\in\mathcal G_\ell.
\end{aligned}
\end{align}
This structure can be summarized schematically as the following block matrix:
\begin{center}
\begin{tikzpicture}[scale=0.7, transform shape]
    \node at (-1.5,3) {\Large$A=$};
    
    \node[draw, minimum size=2.8em] at (0,6) {\Large$a$};
    \node[draw, minimum size=2.8em] at (1,5) {\Large$a$};
    \node at (2,4) {$\ddots$};
    \node[draw, minimum size=2.8em] at (3,3) {\Large$a$};
    \node[draw, minimum size=2.8em] at (4,2) {\Large$a$};
    \node at (5,1) {$\ddots$};
    \node[draw, minimum size=2.8em] at (6,0) {\Large$a$};
    
    \draw (3.5,-0.5) -- (3.5,6.5);
    \draw (-0.5,2.5) -- (6.5,2.5);

    \node at (2.5,5.5) {\Large$b$};
    \node at (0.5,3.5) {\Large$b$};
    \node at (5.5,1.5) {\Large$b$};
    \node at (4.5,0.5) {\Large$b$};
    \node at (5,4.5) {\Large$-b$};
    \node at (1.5,1) {\Large$-b$};

    % 왼쪽 bracket
    \draw[thick] (-0.8,6.5) -- (-0.8,-0.5);
    \draw[thick] (-0.8,6.5) -- (-0.6,6.5);
    \draw[thick] (-0.8,-0.5) -- (-0.6,-0.5);

    % 오른쪽 bracket
    \draw[thick] (6.8,6.5) -- (6.8,-0.5);
    \draw[thick] (6.8,6.5) -- (6.6,6.5);
    \draw[thick] (6.8,-0.5) -- (6.6,-0.5);

    \draw [decorate,decoration={brace,amplitude=5pt,raise=1ex}]
    (-0.5,6.5) -- (3.4,6.5) node[midway,yshift=2em]{\Large$\mathcal G^{(0)}$};
    \draw [decorate,decoration={brace,amplitude=5pt,raise=1ex}]
    (3.6,6.5) -- (6.5,6.5) node[midway,yshift=2em]{\Large$\mathcal G^{(\pi)}$};
\end{tikzpicture}
\end{center}
This leads to the following proposition, which characterizes the eigenvalues of $-J(\Theta^\infty)$.
\begin{proposition}[Local Stability of Antipodal States]\label{prop_eigval_antipodal}
    The Laplacian matrix $L=D_A-A$ of the matrix $A$ defined in \eqref{A_def} has the eigenvalues as follows:
    \begin{align*}
        &\lambda=0~~\mbox{(simple)},\quad\lambda=-bN,~~\mbox{(simple)},\\
        &\lambda=b\left(2\sum_{m\in\mathcal G^{(x)}}|\mathcal G_m|-N\right),~~\mbox{with multiplicity}~~\sum_{m\in\mathcal G^{(x)}}|\mathcal G_m|-1,\\
        &\lambda=(a-b)|\mathcal G_m|+b\bigg(\sum_{\tilde m\in\mathcal G^{(0)}}|\mathcal G_{\tilde m}|-\sum_{\tilde m\in\mathcal G^{(\pi)}}|\mathcal G_{\tilde m}|\bigg)\left(2\mathbf{1}_{\mathcal G^{(0)}}(m)-1\right),\\
        &\mbox{with multiplcity}~|\mathcal G_m|-1,~~\forall~m\in[M],
    \end{align*}
\end{proposition}
The proof of \Cref{prop_eigval_antipodal} is provided in \Cref{sec_appendix}. \Cref{prop_eigval_antipodal} shows that antipodal configurations are generically unstable on block-structured networks.
Any exceptional stable case ($M=2$) is highly constrained and will be treated separately in \Cref{sec_special}.
% \begin{lemma}\label{lem_nec_cond}
%     Suppose that the point $x^*$ is an equilibrium to $\dot x=f(x)$ and the Jacobian matrix $J$ at $x^*$ is symmetric. If there exists a positive diagonal entry in the Jacobian matrix at $x^*$, then the equilibrium is unstable.
% \end{lemma}
% \begin{proof}
%     Without loss of generality, assume that $i$-th diagonal element is positive. By the min-max theorem, we have
%     \begin{align*}
%         \lambda_1=\max_{\|v\|=1}v^\intercal J(x^*)v\ge e_i^\intercal J(x^*)e_i=\big[J(x^*)\big]_{ii}>0,
%     \end{align*}
%     which implies the unstability of $x^*$.
% \end{proof}
%%%%%%%%%%%%%%%%%%%%%%%%%%%%%%%%%%%%%%%%%%%%%%%%%%%%%%%%%%%%%%%%%%%%%%%%%%%%%%%%%%%%%%%%%%%%%%%
\subsection{Kuramoto dynamics on locally excitatory-globally inhibitory networks}\label{subsec_legi}
We consider a signed network $K=(\kappa_{ij})_{i,j\in[N]}\in\{-p,1\}^{N\times N}$ with a circular band structure, where $p>0$ controls the relative strength of inhibition versus excitement. Each node $i$ is positively connected positively with weight $1$ to its $2W$ nearest neighbors along the ring, and negatively connected with weight $-p$ to all other nodes. To distinguish this network from the complete graph while maintaining connectivity, we take
\begin{align*}
    1\le W\le\frac{N-(N\mod{2})}{2}-1.
\end{align*}
More precisely, the adjacency matrix $K=(\kappa_{ij})$ is defined by
\begin{align*}
    \kappa_{ij}=
    \begin{dcases}
        1,\quad&\mbox{if}~~d_N(i,j)\le W,\\
        -p,\quad&\mbox{otherwise},
    \end{dcases},\quad\forall~i,j\in[N],
\end{align*}
where the index distance $d_N$ is given by
\begin{align*}
    d_N(i,j)=\min\Big\{|i-j|,N-|i-j|\Big\}.
\end{align*}
Unlike weakly structurally balanced networks, this network does not admit a partition into positive clusters separated by negative edges. Instead, it provides a homogeneous and highly symmetric setting to investigate the interplay between local excitation and long-range inhibition.

Owing to this translational symmtery, the Kuramoto dynamics on the circular band network admits \textit{rotating-wave} equilibria of the form
\begin{align}\label{rot_wave}
    \theta_j^\infty=\frac{2\pi mj}{N},\quad\forall~j\in[N],
\end{align}
for any integer $0\le m<N$. Due to the symmetry of $\{\theta_i^\infty\}_{i\in[N]}$ and of the coupling matrix $K$, the configuration \eqref{rot_wave} is an equilibrium of
\begin{align}\label{Ad_KM}
    \dot\theta_i=\frac{1}{N}\sum_{j=1}^N\kappa_{ij}\sin(\theta_j-\theta_i),\quad\forall~i\in[N].
\end{align}
The Jacobian matrix at \eqref{rot_wave} is given by
\begin{align*}
    \Big[{\bf J}_{\Theta}(\theta_1^\infty,\cdots,\theta_N^\infty)\Big]_{ij}=
    \begin{dcases}
        -\frac{1}{N}\sum_{\substack{k=1\\k\ne i}}^N\kappa_{ik}\cos\left(\frac{2\pi m(k-i)}{N}\right),\quad&\mbox{if}~~i=j,\\
        \frac{1}{N}\kappa_{ij}\cos\left(\frac{2\pi m(j-i)}{N}\right),&\mbox{otherwise}.
    \end{dcases}
\end{align*}
Due to the translational symmetry of both the network and the rotating-wave equilibrium, this Jacobian is circulant. Consequently, its eigenvalues can be computed explicitly:
\begin{align*}
    \lambda_k&=-2(1+p)\sum_{j=1}^W\cos\left(\frac{2\pi mj}{N}\right)\left(1-\cos\left(\frac{2\pi kj}{N}\right)\right)+p\sum_{j=1}^{N}\cos\left(\frac{2\pi mj}{N}\right)\left(1-\cos\left(\frac{2\pi kj}{N}\right)\right),
\end{align*}
for each $k=0,\cdots,N-1$. For notational convenience, define
\begin{align*}
    S_J(m,k):=\sum_{j=1}^J\cos\left(\frac{2\pi mj}{N}\right)\left(1-\cos\left(\frac{2\pi kj}{N}\right)\right),\quad\forall~J=1,\cdots,N.
\end{align*}

For the equilibrium \eqref{rot_wave} to be locally stable, a necessary condition is $\lambda_k\le0$ for all $k=0,\cdots,N-1$. Moreover, this condition is sufficient provided that the equality does not occur. In what follows, we focus on the necessary condition. The corresponding strict-inequality condition will be noted at the end. Equivalently, $\lambda_k\le0$ can be written as
\begin{align}\label{lam_cond}
    S_N(m,k)\le\frac{2(1+p)}{p}S_W(m,k),\quad k=0,\cdots,N-1.
\end{align}
Using the symmetry of the cosine function, the left hand side of \eqref{lam_cond} can be computed explicitly:
\begin{align*}
    S_N(m,k)&=N\mathbf{1}_{m\equiv0\mod{N}}-\frac{N}{2}\left(\mathbf{1}_{m+k\equiv0\mod{N}}+\mathbf{1}_{m-k\equiv0\mod{N}}\right)\\
    &=\begin{dcases}
        N(1-\mathbf{1}_{k\equiv0\mod{N}}),\quad&\mbox{if}~~m=0,\\
        -\frac{N}{2},&\mbox{if}~~m\ne0~\mbox{and}~m\in\{k,N-k\},\\
        0,&\mbox{otherwise.}
    \end{dcases}
\end{align*}
Hence, for $p\in(0,\infty)$, the stability condition \eqref{lam_cond} may impose a nontrivial upper bound on $p$ in the case $m=0$, whereas a nontrivial lower bound may arise when $m\neq 0$.

\paragraph{Case 1~($m=0$):}
This case corresponds to the complete synchronization. In this setting, the relation \eqref{lam_cond} reduces to
\begin{align*}
    \sum_{j=1}^N\left(1-\cos\left(\frac{2\pi kj}{N}\right)\right)\le2\left(1+\frac{1}{p}\right)\sum_{j=1}^W\left(1-\cos\left(\frac{2\pi kj}{N}\right)\right),
\end{align*}
for all $k=0,\cdots,N-1$. For each fixed $k$, since all summands are nonnegative, there exists a maximum admissible value $p_k^*(W)$ for which the relation \eqref{lam_cond} holds. Moreover, $p_k^*(W)$ is monotone increasing in $W$. Define
\begin{align*}
    p^*(W):=\min_{0\le k\le N-1}p_k^*(W),\quad\forall~0\le W\le\left\lfloor\frac{N-1}{2}\right\rfloor,
\end{align*}
which is also monotone increasing. Consequently, the necessary condition for local stability is given by
\[0<p\le p^*(W).\]

\paragraph{Case 2 ($m\ne0$):}
This case corresponds to an $m$-twist rotating-wave on the ring (equivalently, to $N-m$).

If $k\notin\{m,N-m\}$, then $S_N(m,k)=0$. Thus, the lower bound of $p$, denoted by $p_{*k}(m,W)$, is given by
\begin{align*}
    p_{*k}(m,W)=
    \begin{dcases}
        0,\quad&\mbox{if}~~S_W(m,k)\ge0,\\
        +\infty,&\mbox{if}~~S_W(m,k)<0.
    \end{dcases}
\end{align*}
If $k=m$ or $k=N-m$, then $S_N(m,k)=-\frac{N}{2}$, and hence
\begin{align*}
    p_{*k}(m,W)=
    \begin{dcases}
        0,\quad&\mbox{if}~~S_W(m,k)\ge0,\\
        \left(-\frac{N}{4S_W(m,k)}-1\right)^{-1},&\mbox{if}~~-\frac{N}{4S_W(m,k)}-1>0,\\
        +\infty,&\mbox{otherwise.}
    \end{dcases}
\end{align*}
Therefore, if there exists any $k\in\{0,\cdots,N-1\}$ such that either (i) $k\in\{m,N-m\}$ and $S_W(m,k)<0$, or (ii) $S_W(m,k)<-\frac{N}{4}$, then no positive $p$ satisfies the relation \eqref{lam_cond} for the given $m$ and $W$. Otherwise, defining
\begin{align*}
    p_*(m,W):=\max_{0\le k\le N-1}p_{*k}(m,W),\quad\forall~0\le W\le\left\lfloor\frac{N-1}{2}\right\rfloor,
\end{align*}
we obtain the necessary condition
\[p_*(m,W)\le p<\infty.\]

As an example, the summands in $S_W(m,k)$ satisfy
\begin{align*}
    \cos\left(\frac{2\pi mj}{N}\right)\left(1-\cos\left(\frac{2\pi kj}{N}\right)\right)\ge0,\quad1\le j\le\frac{N}{4m},
\end{align*}
Hence, for $W\le\frac{N}{4m}$, the relations \eqref{lam_cond} admits a nonempty admissible range of $p$. This theoretical insight is confirmed numerically in \Cref{fig_circulant}.
\begin{figure}[h]
    \centering
    \begin{subfigure}[b]{0.24\textwidth}
        \begin{overpic}[width=\textwidth]{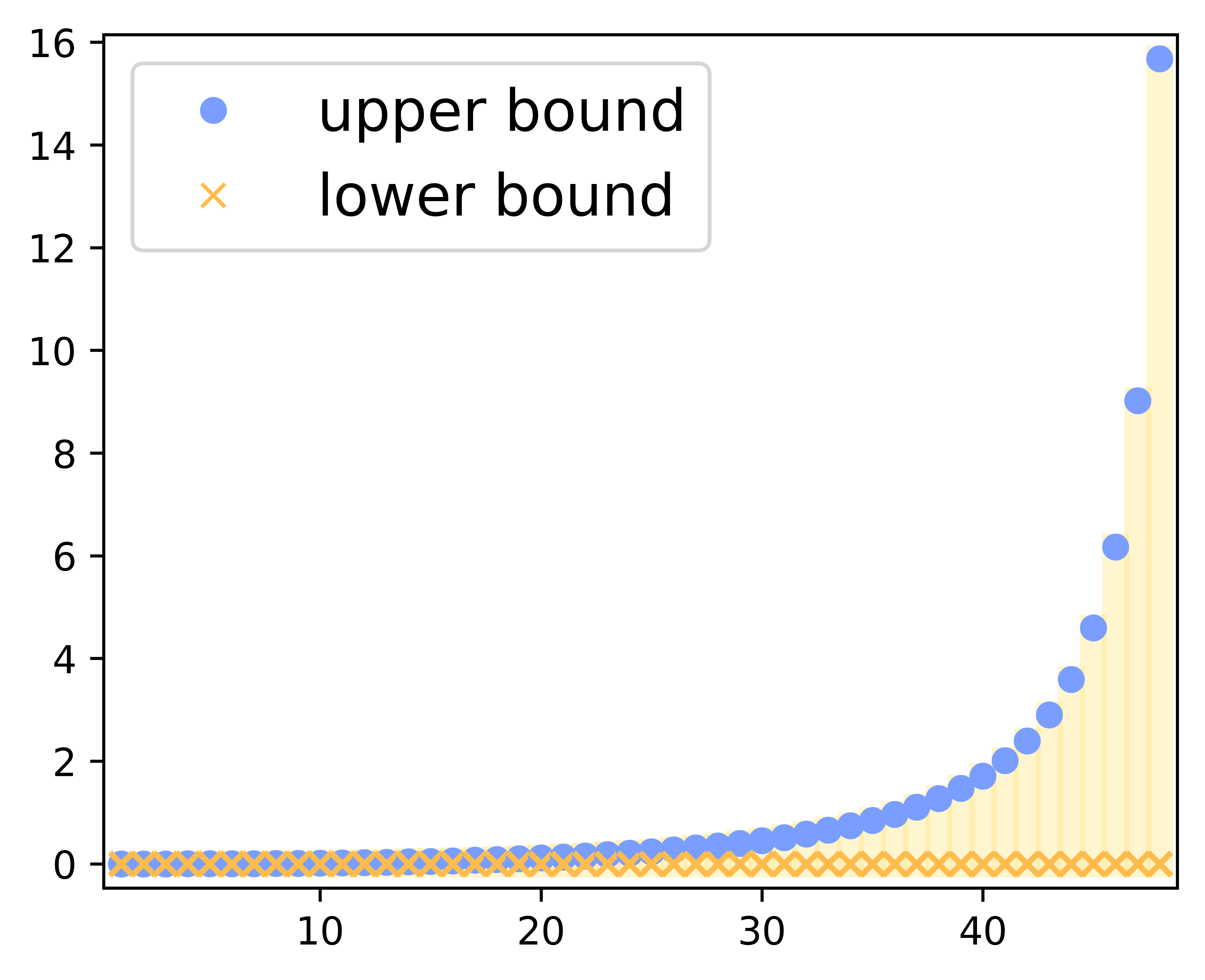}
            \put(-5,42){\small$p$}
            \put(50,-3){\small$W$}
        \end{overpic}
        \caption{$m=0$}
    \end{subfigure}
    \begin{subfigure}[b]{0.24\textwidth}
        \begin{overpic}[width=\textwidth]{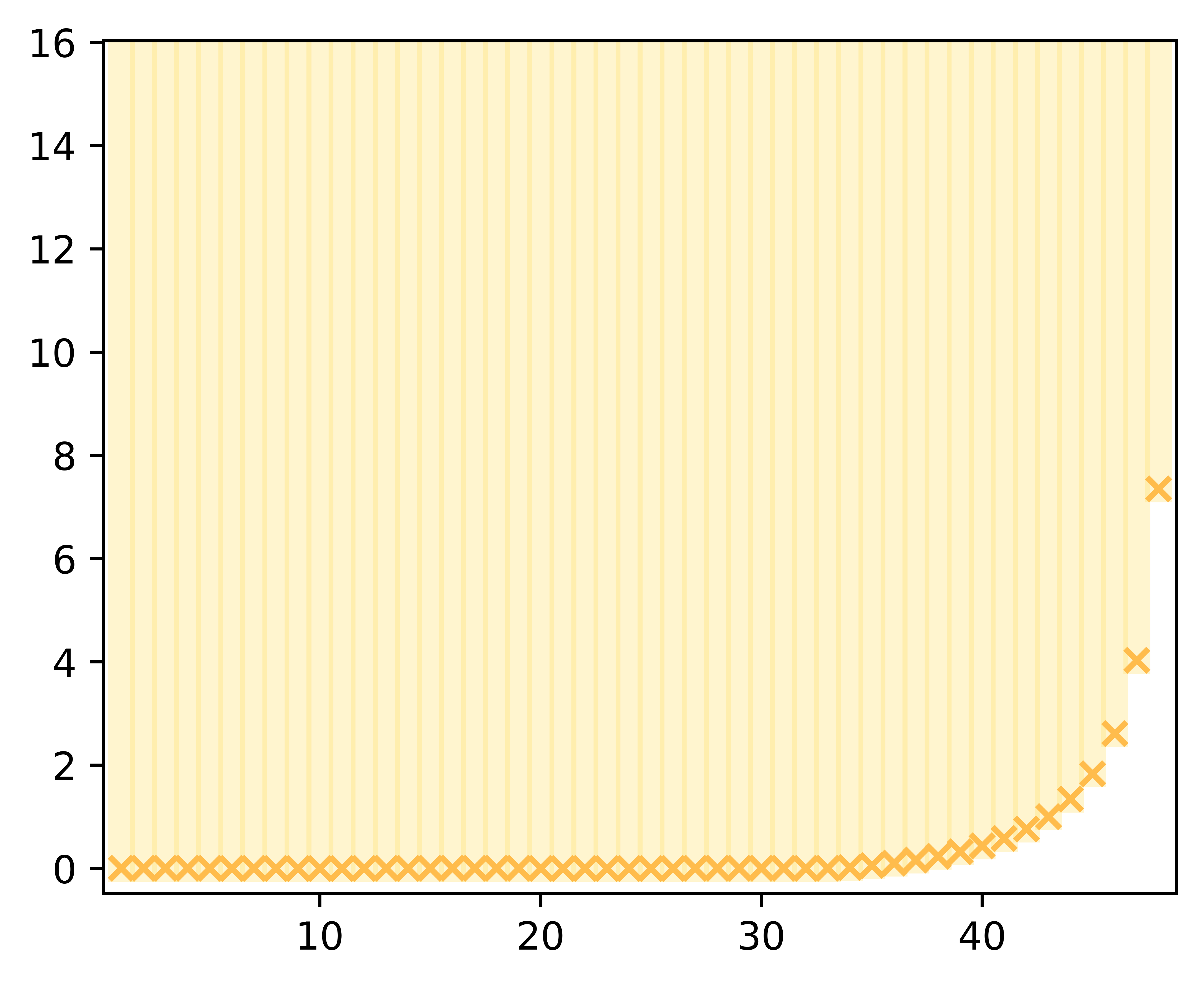}
            
        \end{overpic}
        \caption{$m=1$}
        % \label{sub_}
    \end{subfigure}
    \begin{subfigure}[b]{0.24\textwidth}
        \begin{overpic}[width=\textwidth]{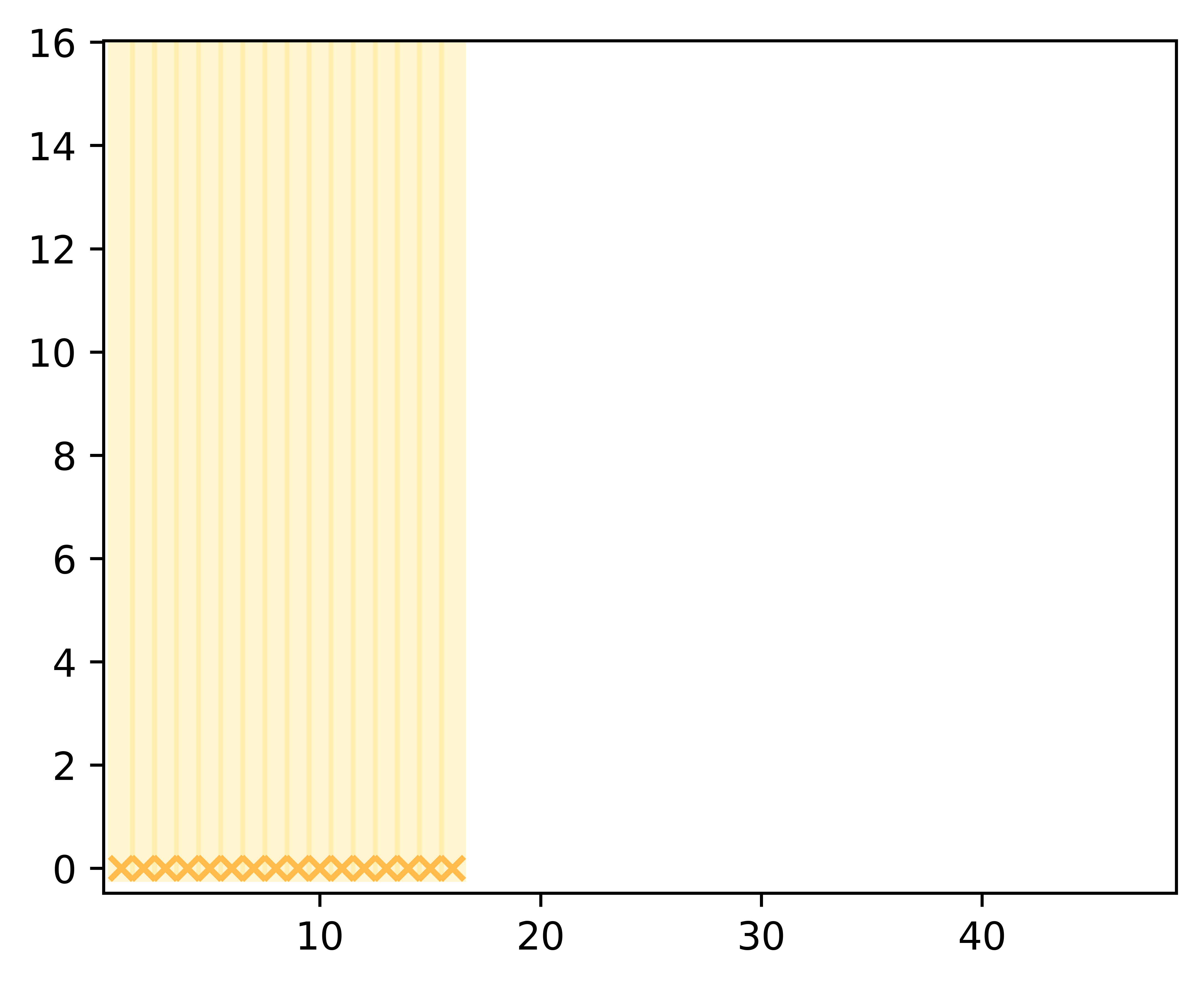}
            
        \end{overpic}
        \caption{$m=2$}
    \end{subfigure}
    \begin{subfigure}[b]{0.24\textwidth}
        \begin{overpic}[width=\textwidth]{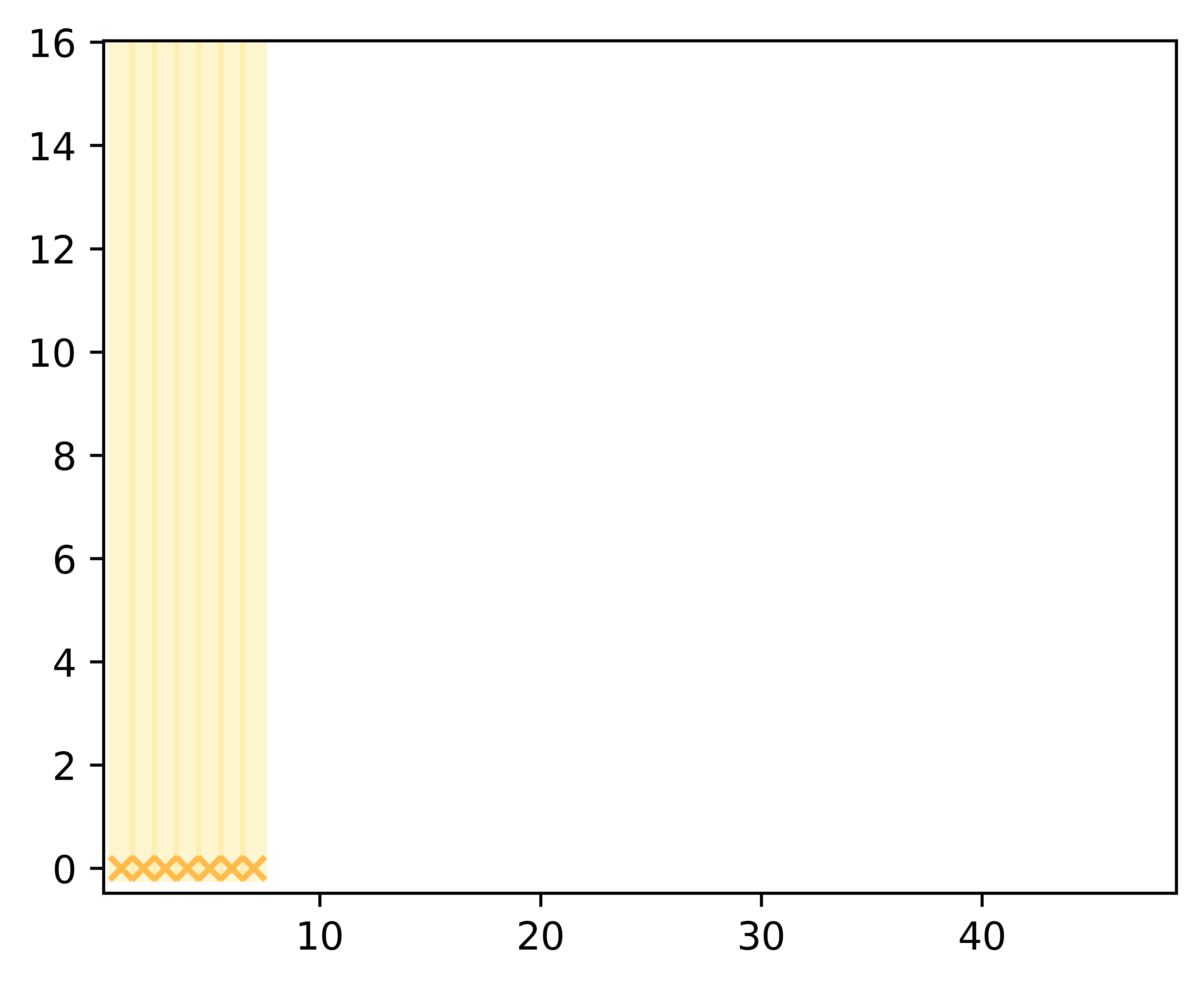}
            
        \end{overpic}
        \caption{$m=4$}
    \end{subfigure}
    \caption{The admissible range of $p$ to satisfy \eqref{lam_cond}. Blue dots and yellow crosses indicate the upper bound and lower bounds, respectively, and the yellow line represents the admissible set of $p$. The absence of markers implies that the admissible set is empty. As demonstrated, the admissible range of $W$ decreases with increasing $m$, suggesting a roughly inverse relationship between the two parameters. All simulations are based on $N=100$.}
    \label{fig_circulant}
\end{figure}

Taken together, the above analysis highlights intrinsic limitations of static signed networks in supporting stable nontrivial phase-locked states.
%%%%%%%%%%%%%%%%%%%%%%%%%%%%%%%%%%%%%%%%%%%%%%%%%%%%%%%%%%%%%%%%%%%%%%%%%%%%%%%%%%%%%%%%%%%%%%%
%%%%%%%%%%%%%%%%%%%%%%%%%%%%%%%%%%%%%%%%%%%%%%%%%%%%%%%%%%%%%%%%%%%%%%%%%%%%%%%%%%%%%%%%%%%%%%%
\section{Complete synchronization and antipodal-states in adaptive Kuramoto network}\label{sec_special}
In this section, we provide the sufficient framework of initial data and parameters under which complete synchronization or antipodal-states emerge in the adaptive Kuramoto networks,
\begin{align}\label{KM_0}
\begin{aligned}
    \begin{dcases}
        \dot\theta_i=-\frac{1}{N}\sum_{j=1}^N\kappa_{ij}\sin(\theta_i-\theta_j+\alpha),\quad\forall~i\in[N],\\
        \dot\kappa_{ij}=-\varepsilon\big(\sin(\theta_i-\theta_j+\beta)+\kappa_{ij}\big),\quad\forall~i,j\in[N].
    \end{dcases}
\end{aligned}
\end{align}
% Simply, we can check
% \begin{align*}
%     \lim_{t\to\infty}\kappa_{ii}(t)=-\sin\beta,\quad\forall~i=1,\cdots,N.
% \end{align*}

We use the method of timescale separation to study multi-timescale systems by analyzing the asymptotic behavior of fast variables and incorporating it into the slow dynamics. This approach is directly motivated by Fenichel’s theorem, which guarantees that a decomposition of fast and slow subsystems is possible in many cases.

First, we consider an invariant set in the adaptive Kuramoto networks \eqref{KM_0}. For constants $c\in(0,\pi)$ and $\delta\in[0,1]$, we define
\begin{align}\label{inv_set}
\begin{aligned}
    \mathcal A_{c,\delta}:=\Big\{(\Theta,K)~:~&\exists \mbox{ a partition }\{\mathcal N_1,\mathcal N_2\}\mbox{ of }[N]\mbox{ such that }\\
    &\theta_i\in[0,c],~\forall~i\in\mathcal N_1,~\theta_i\in[\pi,\pi+c],~\forall~i\in\mathcal N_2,\\
    &\kappa_{ij}\in[\delta,1],~\forall~(i,j)\in\mathcal N_1^2\cup\mathcal N_2^2,\\
    &\kappa_{ij}\in[-1,-\delta],~\forall~(i,j)\in\mathcal N_1\times\mathcal N_2\cup\mathcal N_2\times\mathcal N_1\Big\}.
\end{aligned}
\end{align}
Here, $c$ determines the phase distribution range within each group, while $\delta$ regulates the strength and polarity of interactions. Inherently, any network determined by $K$ in the set $\mathcal A_{c,\delta}$ is structural balanced. In the following lemma, we provide a sufficient conditions for constants $c,\delta$ to satisfy for $A_{c,\delta}$ to be an invariant set.
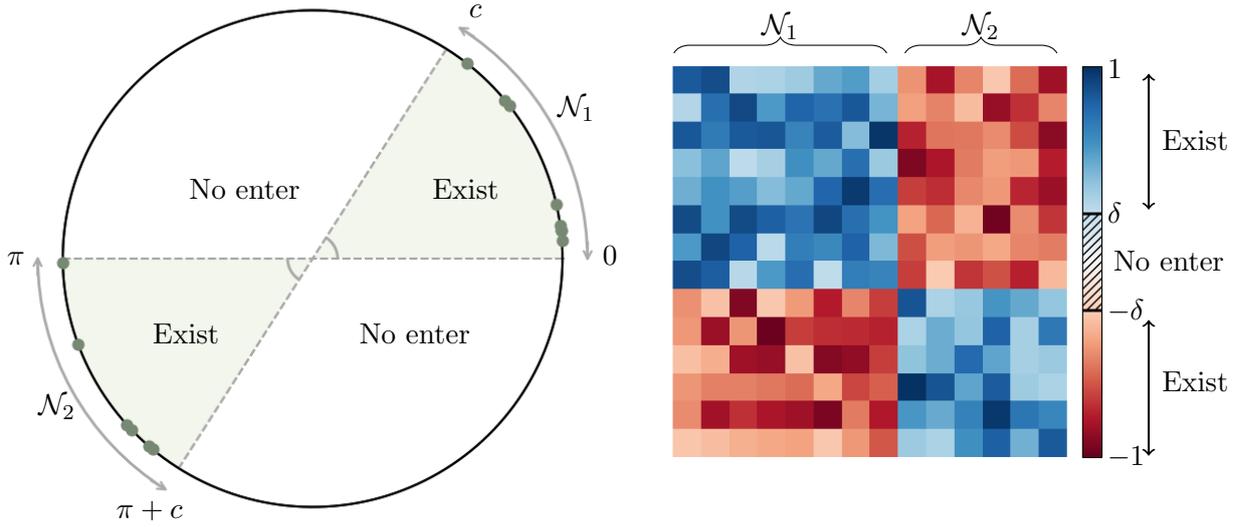
\begin{figure}
    \centering
    \begin{overpic}
        [width=\textwidth]{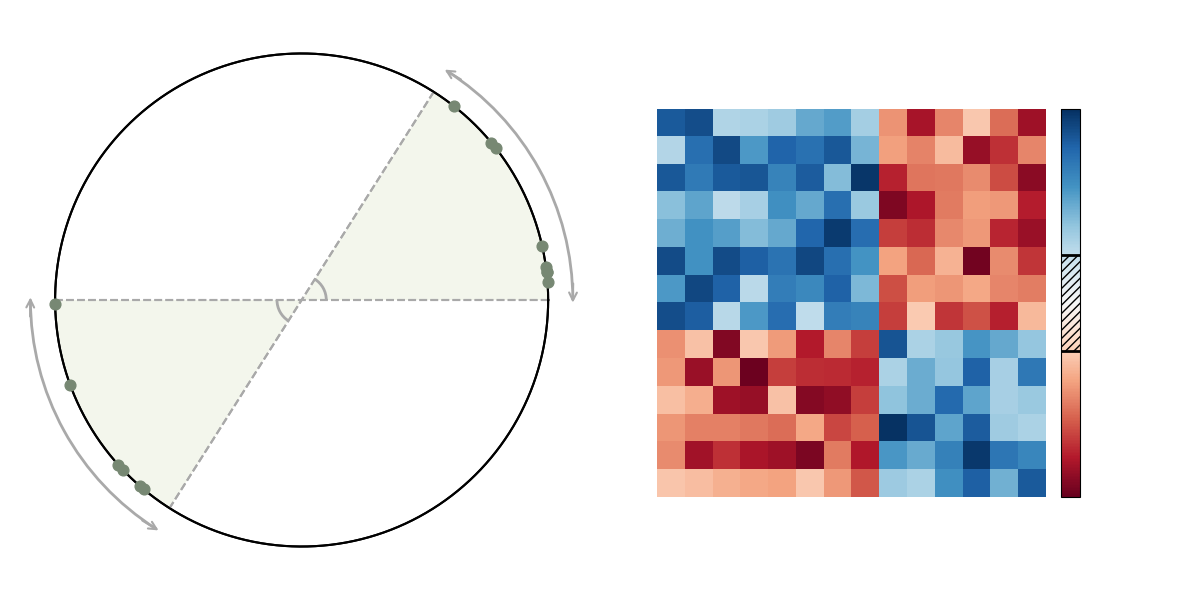}
        \put(38,45){$c$}
        \put(49,24.5){$0$}
        \put(0,24.5){$\pi$}
        \put(9,3.5){$\pi+c$}
        \put(45.3,36.8){$\mathcal N_1$}
        \put(2.6,12.1){$\mathcal N_2$}
        \put(90.5,28){$\delta$}
        \put(90.5,40){$1$}
        \put(90.5,7.9){$-1$}
        \put(90.5,20){$-\delta$}
        \put(93,29){
            \begin{tikzpicture}
              \draw[<->, thick] (0,0) -- (0,1.8);
            \end{tikzpicture}
          }
        \put(93,8.6){
            \begin{tikzpicture}
              \draw[<->, thick] (0,0) -- (0,1.8);
            \end{tikzpicture}
          }
        \put(54,42){
            \begin{tikzpicture}
              \draw [decorate, decoration={brace, amplitude=6pt}]
                (0,0) -- (2.8,0)
                node[midway, yshift=10pt] {$\mathcal N_1$};
            \end{tikzpicture}
          }
        \put(73,42){
            \begin{tikzpicture}
              \draw [decorate, decoration={brace, amplitude=6pt}]
                (0,0) -- (2,0)
                node[midway, yshift=10pt] {$\mathcal N_2$};
            \end{tikzpicture}
          }
        \put(15,30){No enter}
        \put(91,24){No enter}
        \put(95,34){Exist}
        \put(95,14){Exist}
        \put(35,30){Exist}
        \put(12,18){Exist}
        \put(29,18){No enter}
    \end{overpic}
    \caption{Schematic illustration of the invariant set $\mathcal A_{c,\delta}$.
    (Left) Phase configuration associated with a partition $\{\mathcal N_1,\mathcal N_2\}$, where oscillators in $\mathcal N_1$ (resp. $\mathcal N_2$) have phases confined to $[0,c]$ (resp. $[\pi,\pi+c]$).
    (Right) Corresponding coupling matrix $K$, where intra-group couplings satisfy $\kappa_{ij}\in[\delta,1]$ and inter-group couplings satisfy $\kappa_{ij}\in[-1,-\delta]$.}
    \label{fig_setA}
\end{figure}
\begin{lemma}\label{Lem_inv}
    For a hyper-parameter $\beta\in(-\pi,0)$ in \eqref{KM_0}, let a constant $c\in[0,\pi/2)$ to satisfy
    \begin{align}\label{delta_star}
        \delta^*(\beta,c):=-\max\Big(\sin(\beta-c),\sin(\beta+c)\Big)>0
    \end{align}
    and choose $\delta\in[0,1)$ such that $\delta<\delta^*(\beta,c)$. 
    % \begin{align}\label{inv_cond}
    % \begin{aligned}
    %     &\mbox{if}~~\beta\in(-\pi,-\pi/2),\quad\sin(\beta-c)+\delta<0,\\
    %     &\mbox{if}~~\beta\in[-\pi/2,0),\quad\sin(\beta+c)+\delta<0.
    % \end{aligned}
    % \end{align}
    Then, the set $\mathcal A_{c,\delta}$ defined in \eqref{inv_set} is an invariant set for the adaptive network dynamics system \eqref{KM_0}. Moreover, the set $\mathcal A_{c,\delta^*}$
    % for
    % \begin{align}
    %     \delta^*(\beta,c):=-\max\Big(\sin(\beta-c),\sin(\beta+c)\Big)\ge0
    % \end{align}
    is the attractor, i.e., the phase $(\Theta(t),K(t))$ approaches or belongs to $\mathcal A_{c,\delta^*}$ as $t\to\infty$.
\end{lemma}
\begin{proof}
    It is enough to show that configurations are pushed into the set $\mathcal A_{c,\delta}$ at each boundary.

    \vspace{.2cm}
    
    \noindent$\bullet$ {\bf(Boundary at $\theta=0$)} Assume that a phase $\theta_i$ for some $i\in\mathcal N_1$ touches $0$ at time $t$. Then, we have
    \begin{align*}
        &\theta_j-\theta_i\in[0,c],~~\kappa_{ij}\ge\delta>0,\quad\forall~j\in\mathcal N_1,\\
        &\theta_j-\theta_i\in[\pi,\pi+c],~~\kappa_{ij}\le-\delta<0,\quad\forall~j\in\mathcal N_2.
    \end{align*}
    Therefore, one can deduce
    \begin{align*}
        \dot\theta_i&=\frac{1}{N}\left(\sum_{j\in\mathcal N_1}\kappa_{ij}\sin(\theta_j-\theta_i)+\sum_{j\in\mathcal N_2}\kappa_{ij}\sin(\theta_j-\theta_i)\right)\ge0.
    \end{align*}

    \vspace{.2cm}

    \noindent$\bullet$ {\bf (Boundary at $\theta=c$)} Similarly, we have
    \begin{align*}
        &\theta_j-\theta_i\in[-c,0],~~\kappa_{ij}\ge\delta>0,\quad\forall~j\in\mathcal N_1,\\
        &\theta_j-\theta_i\in[\pi-c,\pi],~~\kappa_{ij}\le-\delta<0,\quad\forall~j\in\mathcal N_2.
    \end{align*}
    Then, the decrease of $\theta_i$ is guaranteed, i.e., $\dot\theta_i\le0$.

    \vspace{.2cm}

    \noindent Likewise, we can show the invariant property of $\mathcal A_{c,\delta}$ at the boundary $\theta=\pi$ and $\theta=\pi+c$.

    \vspace{.2cm}

    \noindent$\bullet$ {\bf (Boundary at $\kappa=\delta$)} We assume $\kappa_{ij}=\delta$ at time $t$ for some $(i,j)\in\mathcal N_1\times\mathcal N_1\cup\mathcal N_2\times\mathcal N_2$. At time $t$, the variable $\kappa_{ij}$ follows
    \begin{align*}
        \dot\kappa_{ij}&=-\varepsilon\cdot\big(\sin(\theta_i-\theta_j+\beta)+\delta\big)\\
        &=-\varepsilon\cdot\Big(\sin(\theta_i-\theta_j)\cos\beta+\cos(\theta_i-\theta_j)\sin\beta+\delta\Big).
    \end{align*}
    Note that $\theta_i-\theta_j\in[-c,c]$ and
    \begin{align}\label{abs_inf}
        0\le\big|\sin(\theta_i-\theta_j)\cos\beta\big|\le \big|\sin c\cos\beta\big|,\quad\big|\cos c\sin\beta\big|\le\big|\cos(\theta_i-\theta_j)\sin\beta\big|\le\big|\sin\beta\big|.
    \end{align}
    Now, we split the cases of $\beta\in(-\pi,-\pi/2)$ and $\beta\in[-\pi/2,0)$.

    \vspace{.1cm}
    
    \noindent - \underline{Case 1} ($\beta\in(-\pi,-\pi/2)$): Here, we have 
    \[\sin\beta<0,~~\cos\beta<0.\]
    If $\theta_i-\theta_j\in[-c,0]\subset(-\pi/2,0)$, one has
    \begin{align*}
        \dot\kappa_{ij}&=-\varepsilon\cdot\Big(\big|\sin(\theta_i-\theta_j)\cos\beta\big|-\big|\cos(\theta_i-\theta_j)\sin\beta\big|+\delta\Big)\\
        &\ge-\varepsilon\cdot\Big(-\sin c\cos\beta+\cos c\sin\beta+\delta\Big)=-\varepsilon\cdot\Big(\sin(\beta-c)+\delta\Big)>0,
    \end{align*}
    where we use the relation \eqref{abs_inf} and condition $\delta<\delta^*$ in two inequalities, respectively. Likewise, if $\theta_i-\theta_j\in(0,c]\subset(0,\pi/2)$, we can derive
    \begin{align*}
        \dot\kappa_{ij}&=-\varepsilon\cdot\Big(-\big|\sin(\theta_i-\theta_j)\cos\beta\big|-\big|\cos(\theta_i-\theta_j)\sin\beta\big|+\delta\Big)\\
        &\ge-\varepsilon\cdot\Big(\cos c\sin\beta+\delta\Big)>0,
    \end{align*}
    where we use
    \begin{align*}%\label{sincos}
        0>\sin(\beta-c)+\delta\ge\sin\beta\cos c+\delta
    \end{align*}
    in the last inequality. 
    
    \vspace{.1cm}

    \noindent - \underline{Case 2} ($\beta\in[-\pi/2,0)$): If $\theta_i-\theta_j\in[-c,0]$, we use the relation \eqref{abs_inf} and
    \[0>\sin(\beta+c)+\delta\ge\sin\beta\cos c+\delta\]
    to obtain
    \begin{align*}
        \dot\kappa_{ij}&=-\varepsilon\cdot\Big(-\big|\sin(\theta_i-\theta_j)\cos\beta\big|-\big|\cos(\theta_i-\theta_j)\sin\beta\big|+\delta\Big)\\
        &\ge-\varepsilon\cdot\Big(\cos c\sin\beta+\delta\Big)>0.
    \end{align*}
    Similarly, for $\theta_i-\theta_j\in(0,c]$, we get
    \begin{align*}
        \dot\kappa_{ij}&=-\varepsilon\cdot\Big(\big|\sin(\theta_i-\theta_j)\cos\beta\big|-\big|\cos(\theta_i-\theta_j)\sin\beta\big|+\delta\Big)\\
        &\ge-\varepsilon\cdot\Big(\sin c\cos\beta+\cos c\sin\beta+\delta\Big)=-\varepsilon\cdot\Big(\sin(\beta+c)+\delta\Big)>0.
    \end{align*}

    \vspace{.1cm}

    \noindent From two cases, we conclude that
    \[\dot\kappa_{ij}(t)\ge0\quad\mbox{when}\quad\kappa_{ij}(t)=\delta.\]

    \vspace{.2cm}

    \noindent We can also use the similar argument when $\kappa_{ij}=-\delta$, which leads to $\dot\kappa_{ij}\le0$.

    \vspace{.2cm}

    Now, we are left to show the attracting set $\mathcal A_{c,\delta^*}$. Note that
    \begin{align*}
        \theta_j-\theta_i-\beta\in\big[-c-\beta,c-\beta\big],\quad&\mbox{if}~~(i,j)\in\mathcal N^1\cup\mathcal N^2,\\
        \theta_j-\theta_i-\beta\in[\pi+c-\beta,\pi-c-\beta],&\mbox{otherwise}.
    \end{align*}
    This and the condition $\delta<\delta^*$ imply that the sine term satisfies
    \begin{align*}
        \sin(\theta_j-\theta_i-\beta)\in[\min\big(\sin(-c-\beta),\sin(c-\beta)\big),1],\quad&\mbox{if}~~(i,j)\in\mathcal N^1\cup\mathcal N^2,\\
        \sin(\theta_j-\theta_i-\beta)\in\big[-1,\min\big(\sin(-c-\beta),\sin(c-\beta)\big)\big],&\mbox{otherwise}.
    \end{align*}
    which with the adaptive rule 
    \begin{align*}
        \dot\kappa_{ij}=\varepsilon\Big(\sin(\theta_j-\theta_i-\beta)-\kappa_{ij}\Big)
    \end{align*}
    shows the attracting properties of $\mathcal A_{c,\delta^*}$. This ends the proof.
\end{proof}

Lemma \ref{Lem_inv} guarantees the boundedness of coupling strengths $\kappa_{ij}$, which are suitable for the purpose of bipartite synchronization, without any change of signs in coupling strengths. More precisely, the couplings within each groups, defined by $\mathcal N_1$ and $\mathcal N_2$, lead to aggregation among agents within the same group, while the couplings between groups result in repulsion between agents on $\mathbb T^1$. 

From now on, we consider dynamics under the Lemma \ref{Lem_inv} regime. That is, without loss of generality, we assume that an initial data $(\Theta^0,K^0)$ satisfy
\begin{align}\label{th_rng}
    \theta_i^0\in[0,c],\quad\forall~i\in\mathcal N_1,\quad\theta_i^0\in[\pi,\pi+c],\quad\forall~i\in\mathcal N_2,
\end{align}
and
\begin{align}\label{k_rng}
    \begin{aligned}
        &\kappa_{ij}^0\in[\delta,1],\quad\forall~(i,j)\in\mathcal N_1\times\mathcal N_1\cup\mathcal N_2\times\mathcal N_2,\\
        &\kappa_{ij}^0\in[-1,-\delta],\quad\forall~(i,j)\in\mathcal N_1\times\mathcal N_2\cup\mathcal N_2\times\mathcal N_1.
    \end{aligned}
\end{align}
Above relations \eqref{th_rng} and \eqref{k_rng} hold for all $t\ge0$ by Lemma \ref{Lem_inv}. Based on this bounded relations, we may rewrite the system $(\Theta,K)\leftrightarrow(\tilde\Theta,\tilde K)$ as
\begin{align*}
    \tilde\theta_i=
    \begin{dcases}
        \theta_i,\quad\forall~i\in\mathcal N_1,\\
        \theta_i-\pi,\quad\forall~i\in\mathcal N_2,
    \end{dcases}
    \quad\tilde\kappa_{ij}=
    \begin{dcases}
        \kappa_{ij},\quad\forall~(i,j)\in\mathcal N_1^2\cup\mathcal N_2^2,\\
        -\kappa_{ij},\quad\forall~(i,j)\in\mathcal N_1\times\mathcal N_2\cup\mathcal N_2\times\mathcal N_1.
    \end{dcases}
\end{align*}
One can easily check that $\tilde\theta_i\in[0,c]$ and $\kappa_{ij}\in[\delta,1]$ for any $i,j\in[N]$ and
\begin{align*}
    \begin{aligned}
        \begin{dcases}
            \frac{d}{dt}\tilde\theta_i=\frac{1}{N}\sum_{j=1}^N\tilde\kappa_{ij}\sin(\tilde\theta_j-\tilde\theta_i),\\
            \frac{d}{dt}\tilde\kappa_{ij}=-\varepsilon\big(\sin(\tilde\theta_i-\tilde\theta_j+\beta)+\tilde\kappa_{ij}\big),
        \end{dcases}
    \end{aligned}
\end{align*}
which is the same structure as the original model \eqref{KM_0}. Thus, showing the complete synchronization of $(\tilde\Theta,\tilde K)$ is equivalent to showing the bipartite synchronization of $(\Theta,K)$ appearing under the Lemma \ref{Lem_inv} regime. This leads to consider the initial data in $\mathcal A_{c,\delta}$ with a trivial partition, i.e., $\mathcal N_1=[N]$ and $\mathcal N_2=\varnothing$. Here, we set several notation
\begin{align*}
    \theta_M(t):=\max_{i\in[N]}\theta_i(t),\quad\theta_m(t):=\min_{i\in[N]}\theta_i(t),\quad\mathcal D(t):=\theta_M(t)-\theta_m(t),
\end{align*}
of the maximum and minimum phase and phase diameter.
\begin{theorem}\label{Thm_bi_sync}
    Suppose that an initial data $(\Theta^0,K^0)$ satisfy
    \begin{align}\label{thm1_cond}
        \mathcal D^0<\min(\pi+\beta,|\beta|)<\frac{\pi}{2}\quad\mbox{and}\quad\min_{i,j\in[N]}\kappa_{ij}^0>0.
    \end{align}
    Then, the solution $(\Theta,K)$ to \eqref{KM_0} satisfy
    \[\lim_{t\to\infty}\mathcal D(t)=0\quad\mbox{and}\quad\lim_{t\to\infty}\kappa_{ij}(t)=-\sin\beta,\quad\forall~i,j\in[N].\]
\end{theorem}
\begin{proof}
    First, we consider the parameter $\delta $ in Lemma \ref{Lem_inv} with $c=\mathcal D^0$. Note that the condition \eqref{thm1_cond} implies that $\delta^*(\beta,\mathcal D^0)<0$.
    % \begin{align*}
    %     \begin{aligned}
    %         &\mbox{if}~~\beta\in(-\pi,-\pi/2),\quad&&-1<\sin(\beta-\mathcal D^0)<0,\\
    %         &\mbox{if}~~\beta\in[-\pi/2,0),\quad&&-1<\sin(\beta+\mathcal D^0)<0.
    %     \end{aligned}
    % \end{align*}
    So, the existence of $\delta\in[0,1)$ is guaranteed as
    \begin{align*}
        \delta<\min\left(\delta^*(\beta,\mathcal D^0),\min_{i,j\in[N]}\kappa_{ij}^0\right).
    \end{align*}
    Hence, by Lemma \ref{Lem_inv}, we have
    \begin{align}\label{DK_range}
        \begin{aligned}
            0\le\mathcal D(t)\le\mathcal D^0<\frac{\pi}{2},\quad\delta^*(\beta,\mathcal D^0)\le\kappa_{ij}(t)\le1,\quad\forall~i,j\in[N],~~t\ge0.
        \end{aligned}
    \end{align}
    From \eqref{KM_0}, one can compute
    \begin{align}\label{dt_diam}
    \begin{aligned}
        \frac{d}{dt}\mathcal D(t)&=\frac{1}{N}\sum_{j=1}^N\Big(\kappa_{Mj}\sin(\theta_j-\theta_M)-\kappa_{mj}\sin(\theta_j-\theta_m)\Big)\\
        &\le\frac{2\delta}{N}\sum_{j=1}^N\cos\left(\frac{2\theta_j-\theta_M-\theta_m}{2}\right)\sin\left(\frac{\theta_m-\theta_M}{2}\right)\\
        &\le-2\delta\cos\mathcal D^0\sin\left(\frac{\mathcal D(t)}{2}\right),
    \end{aligned}    
    \end{align}
    where first and second inequalities comes from the relations \eqref{DK_range}. By direct calculations, we derive that the solution satisfies
    \begin{align*}
        f\left(\frac{\mathcal D(t)}{2}\right)\le f\left(\frac{\mathcal D^0}{2}\right)e^{-\delta^*(\beta,\mathcal D^0)t\cos \mathcal D^0},
    \end{align*}
    where the real-valued function $f$ is defined by
    \begin{align}\label{def_f}
        f(x):=\csc(x)-\cot(x).
    \end{align}
    Since the function $f$ is monotone increasing on $(0,\pi)$ and $\lim_{x\to0+}f(x)=0$, we derive
    \begin{align}\label{diam_conv}
        \lim_{t\to\infty}\mathcal D(t)=0.
    \end{align}
    Now, it is enough to show the convergence of coupling strengths. The convergence of diameter \eqref{diam_conv} means for any $0<\zeta$, there exists $t^*(\zeta)>0$ such that
    \[|\theta_i(t)-\theta_j(t)|<\zeta,\quad\forall~t\ge t^*(\zeta),~~\forall~i,j\in[N].\]
    Here, we let 
    \begin{align*}
        \zeta=\frac{1}{2}\min(\pi+\beta,-\beta),
    \end{align*}
    which implies the following relation
    \begin{align}\label{kappa_sndw}
        \varepsilon m_{\beta,\zeta}\le\dot\kappa_{ij}(t)+\varepsilon\kappa_{ij}(t)\le \varepsilon M_{\beta,\zeta},\quad\forall~t\ge t^*(\zeta),
    \end{align}
    where the constants are defined as
    \begin{align*}
        M_{\beta,\zeta}&:=\max\Big(\sin(-\zeta-\beta),\sin(\zeta-\beta),\sin(-\beta)\Big),\\
        m_{\beta,\zeta}&:=\min\Big(\sin(-\zeta-\beta),\sin(\zeta-\beta),\sin(-\beta)\Big).
    \end{align*}
    From \eqref{kappa_sndw}, one obtains
    \begin{align*}
        m_{\beta,\zeta}+\left(\kappa_{ij}(t^*(\zeta))-m_{\beta,\zeta}\right)e^{-\varepsilon(t-t^*(\zeta))}\le\kappa_{ij}(t)\le M_{\beta,\zeta}+\left(\kappa_{ij}(t^*(\zeta))-M_{\beta,\zeta}\right)e^{-\varepsilon(t-t^*(\zeta))},
    \end{align*}
    hence,
    \begin{align*}
        m_{\beta,\zeta}\le\lim_{t\to\infty}\kappa_{ij}(t)\le M_{\beta,\zeta}.
    \end{align*}
    Here, we use the continuity of $M_{\beta,\zeta}$ and $m_{\beta,\zeta}$, i.e.,
    \[\lim_{\zeta\to 0}M_{\beta,\zeta}=\lim_{\zeta\to0}m_{\beta,\zeta}=-\sin\beta\]
    to complete the proof.
\end{proof}
\begin{figure}
    \centering
    \begin{subfigure}{0.45\textwidth}
        \centering
        \begin{overpic}[width=\textwidth]{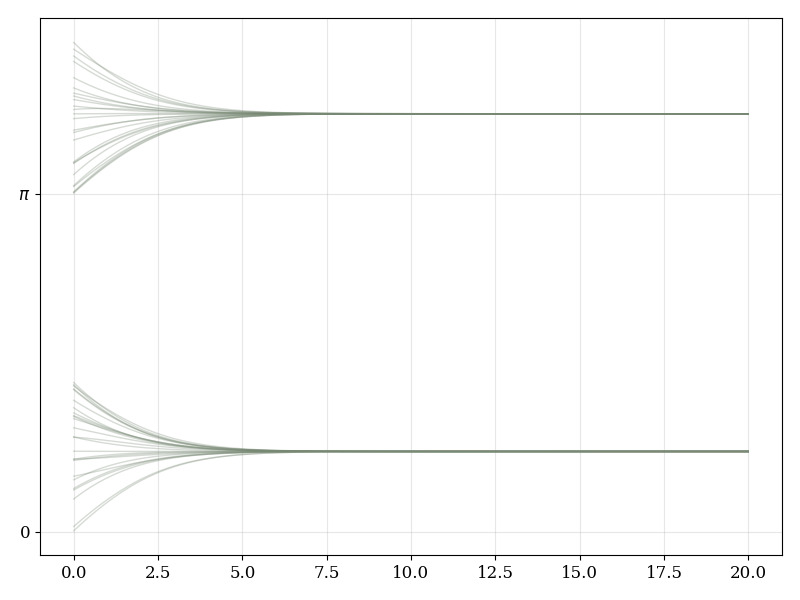}
            \put(-2,38){\small$\theta_i$}
            \put(44,-2){\small time $t$}
        \end{overpic}
        \caption{Evolution of phases}
    \end{subfigure}
    \hspace{.5cm}
    \begin{subfigure}{0.45\textwidth}
        \centering
        \begin{overpic}[width=\textwidth]{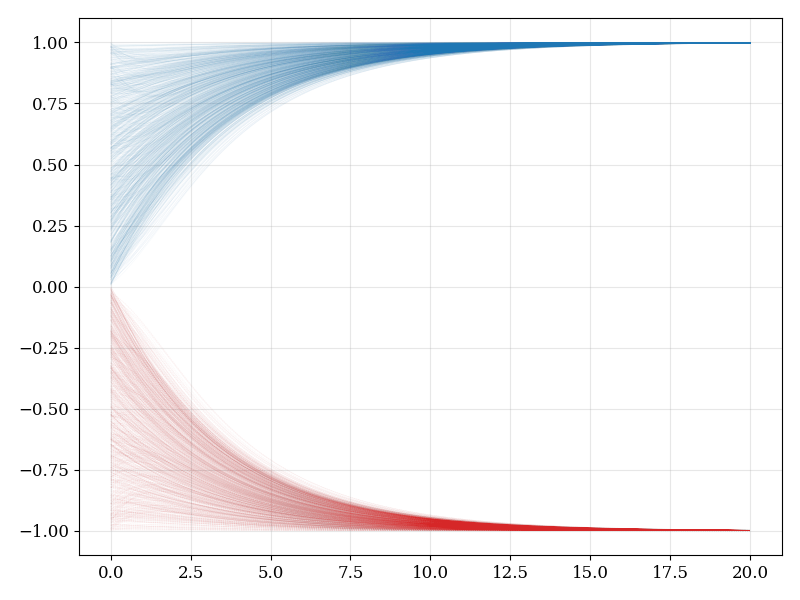}
            \put(47,-2){\small time $t$}
            \put(-3,38){\small$\kappa_{ij}$}
        \end{overpic}
        \caption{Evolution of coupling strengths}
    \end{subfigure}
    \caption{Numeric example with an initial data satisfying the conditions in Theorem \ref{Thm_bi_sync} ($\beta=-\pi/2)$.}
\end{figure}
% For $f(x)=\csc(x)-\cot(x)$, we define $A:(-\pi,0)\times(0,\infty)\times[0,1]$ as follows: if $\beta\in(-\pi,-\pi/2)$,
% \begin{align*}
%     A(\beta,\varepsilon,\kappa_{\min}^0):=\max_{0\le D<\pi+\beta}f(D/2)\exp\left(-\frac{\sin(\beta-D)}{\varepsilon}\ln\left(\frac{\sin(\beta-D)+\kappa_{\min}^0}{\sin(\beta-D)}\right)+\frac{\kappa_{\min}^0}{\varepsilon}\right),
% \end{align*}
% and if $\beta\in(-\pi/2,0)$,
% \begin{align*}
%     A(\beta,\varepsilon,\kappa_{\min}^0):=\max_{0\le D<-\beta}f(D/2)\exp\left(-\frac{\sin(\beta+D)}{\varepsilon}\ln\left(\frac{\sin(\beta+D)+\kappa_{\min}^0}{\sin(\beta+D)}\right)+\frac{\kappa_{\min}^0}{\varepsilon}\right).
% \end{align*}
We now consider a more general setting in which the initial coupling strengths
are not fully separated by sign. Our goal is to understand whether the adaptive
dynamics of the couplings, induced by $\varepsilon>0$, can still guarantee
synchronization within each group. To simplify notation, we define
\[\kappa_{\min}^0 := \min_{i,j\in[N]} \kappa^0_{ij}.\]
\begin{theorem}\label{thm_adap}
    Consider initial data $(\Theta^0,K^0)$ with $\kappa_{\min}^0<0$ and $\mathcal D^0\le\overline{\mathcal D}$ where the constant $\overline{\mathcal D}=\overline{\mathcal D}\left(\beta,\varepsilon,\kappa_{\min}^0\right)$ is defined by
    \begin{align}\label{Dbar}
        \begin{aligned}
            &\overline{\mathcal D}\left(\beta,\varepsilon,\kappa_{\min}^0\right):=\argmax_{0\le\mathcal D\le\min(\pi+\beta,|\beta|)}f\left(\frac{\mathcal D}{2}\right)\exp\left(\frac{\delta^*}{\varepsilon}\ln\left(\frac{\delta^*-\kappa_{\min}^0}{\delta^*}\right)+\frac{\kappa_{\min}^0}{\varepsilon}\right),
        \end{aligned}
    \end{align}
    for the constant $\delta^*=\delta^*(\beta,\mathcal D)$ and the function $f$ are defined in \eqref{delta_star} and \eqref{def_f}, respectively. Then, the solution $(\Theta,K)$ to \eqref{KM_0} satisfies
    \begin{align*}
        \lim_{t\to\infty}\mathcal D(t)=0\quad\mbox{and}\quad\lim_{t\to\infty}\kappa_{ij}(t)=-\sin\beta,\quad\forall~i,j\in[N].
    \end{align*}
\end{theorem}
\begin{proof}
    Set 
    \begin{align*}
        \tilde t:=\sup\{t>0~:~\mathcal D(s)\le\overline{\mathcal D},\quad\forall~s\in(0,t)\}.
    \end{align*}
    By the continuity of $\mathcal D(t)$, there exists $\tilde t>0$. On $t\in[0,\tilde t~]$, one has
    \begin{align*}
        \dot\kappa_{ij}&=\varepsilon(\sin(\theta_j-\theta_i-\beta)-\kappa_{ij})\le\varepsilon(\delta^*-\kappa_{ij}),
    \end{align*}
    which gives an increasing lower bound for each network weight, in particular,
    \begin{align*}
        \kappa_{ij}(t)\ge\delta^*+(\kappa_{ij}^0-\delta^*)e^{-\varepsilon t},\quad\forall~t\le\tilde t.
    \end{align*}
    This leads to
    \begin{align*}
        \kappa_{ij}(t)\ge0,\quad\forall~t\ge\frac{1}{\varepsilon}\ln\left(\frac{\delta^*-\kappa_{\min}^0}{\delta^*}\right)=:\tilde T(\varepsilon,\delta^*,\kappa_{\min}^0).
    \end{align*}
    Note that if $\mathcal D(t)\le\overline{\mathcal D}$ until every network weights gets greater or equal to zero, i.e., until time $\tilde T$, then we have the desired result by Theorem \ref{Thm_bi_sync}. Hence, now we examine how much increment in diameter $\mathcal D(t)$ could occur until time $\tilde T$. Note that if $\mathcal D(t)\le\overline{\mathcal D}$, we have
    \begin{align*}
        \frac{d}{dt}\mathcal D(t)&=\frac{2}{N}\Big(\delta^*+\Big(\kappa_{\min}^0-\delta^*\Big)e^{-\varepsilon t}\Big)\sum_{j=1}^N\cos\left(\frac{2\theta_j-\theta_M-\theta_m}{2}\right)\sin\left(\frac{\theta_m-\theta_M}{2}\right)\\
        &\le-2\Big(\delta^*+\Big(\kappa_{\min}^0-\delta^*\Big)e^{-\varepsilon t}\Big)\sin\left(\frac{\mathcal D(t)}{2}\right).
    \end{align*}
    Again using the function $f$ defined in \eqref{def_f}, one gets
    \begin{align*}
        f\left(\frac{\mathcal D(t)}{2}\right)\le f\left(\frac{\mathcal D^0}{2}\right)\exp\left(-\delta^*t+\frac{1}{\varepsilon}\left(\delta^*-\kappa_{\min}^0\right)(1-e^{-\varepsilon t})\right)
    \end{align*}
    if $\mathcal D(s)\le\overline{\mathcal D}$ for each $s\in(0,t)$. Note that the exponent is positive for all $t\in[0,\tilde t~]$, hence the definition of $\overline{\mathcal D}$ \eqref{Dbar} implies that
    \begin{align*}
        \mathcal D(t)\le\overline{\mathcal D},\quad\forall~t\in[0,\tilde t~].
    \end{align*}
    This completes the proof.
\end{proof}
\begin{remark}
    Figures \ref{fig_cond_2d} and \ref{fig_cond_3d} provide a quantitative illustration of the critical diameter $\overline{\mathcal D}$ defined in \Cref{thm_adap}. In particular, they show how the maximal admissible initial phase diameter guaranteeing convergence to the (antipodal) synchronized state depends on the adaptation rate $\varepsilon$, the phase-lag parameter $\beta$, and the minimal initial coupling strength $\kappa_{\min}$.

    \Cref{fig_cond_2d} visualizes the critical diameter $\overline{\mathcal D}$ over the $(\varepsilon,\kappa_{\min}^0)$-plane for several values of the phase-lag parameter $\beta$. Across all panels, $\overline{\mathcal D}$ is close to zero for $\varepsilon\approx 0$, indicating that without coupling adaptivity there is essentially no nontrivial guarantee of convergence for spread initial phases. As $\varepsilon$ increases, $\overline{\mathcal D}$ increases, showing that faster adaptation enlarges the admissible set of initial phase configurations. Moreover, for fixed $\varepsilon$, the critical diameter decreases as $\kappa_{\min}^0$ becomes more negative, i.e., stronger initial inhibitory couplings reduce the guaranteed basin of attraction. Comparing the three panels, the overall level of $\overline{\mathcal D}$ decreases as $\beta$ moves from $-0.5\pi$ toward $0$, demonstrating a pronounced dependence of the sufficient bound on the phase-lag parameter $\beta$.

    \Cref{fig_cond_3d} provides a three-dimensional visualization of the critical diameter $\overline D$ as a function of the phase-lag parameter $\beta$ and the adaptation rate $\varepsilon$, for several fixed values of the minimal initial coupling $\kappa_{\min}^0$. The surfaces are symmetric with respect to $\beta$ and attain their maximal values near $\beta=-\pi/2$, while $\overline D$ vanishes as $\beta$ approaches $0$ or $-\pi$. Moreover, the plots show that $\overline D$ is identically zero at $\varepsilon=0$ and becomes strictly positive for any $\varepsilon>0$, demonstrating that coupling adaptivity is essential for generating a nontrivial basin of attraction. As $\varepsilon$ increases, the critical diameter grows, with only marginal changes observed for larger values of $\varepsilon$. Comparing panels (a)--(d), stronger initial inhibition (more negative $\kappa_{\min}^0$) uniformly reduces the achievable values of $\overline D$ across the entire parameter range.
\end{remark}

\begin{figure}[h]
    \centering
    \begin{overpic}
        [width=\textwidth]{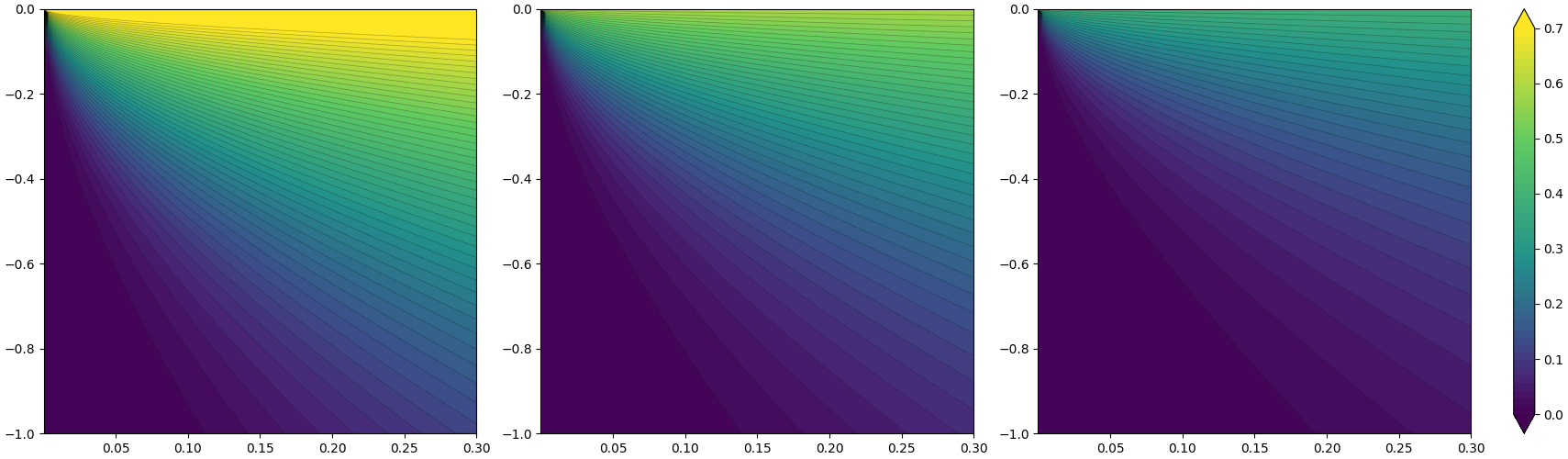}
        \put(12,30){$\beta=-0.5\pi$}
        \put(42,30){$\beta=-0.375\pi$}
        \put(74,30){$\beta=-0.25\pi$}
        \put(16,-1.5){$\varepsilon$}
        \put(-4,15){$\kappa_{\min}^0$}
        \put(101,15){$\overline{\mathcal D}$}
    \end{overpic}
    \caption{Contour plots of the critical diameter $\overline{\mathcal D}$ as a function of the adaptation rate $\varepsilon$ and the minimal initial coupling $\kappa_{\min}^0$, shown for several fixed values of $\beta$.}
    \label{fig_cond_2d}
\end{figure}
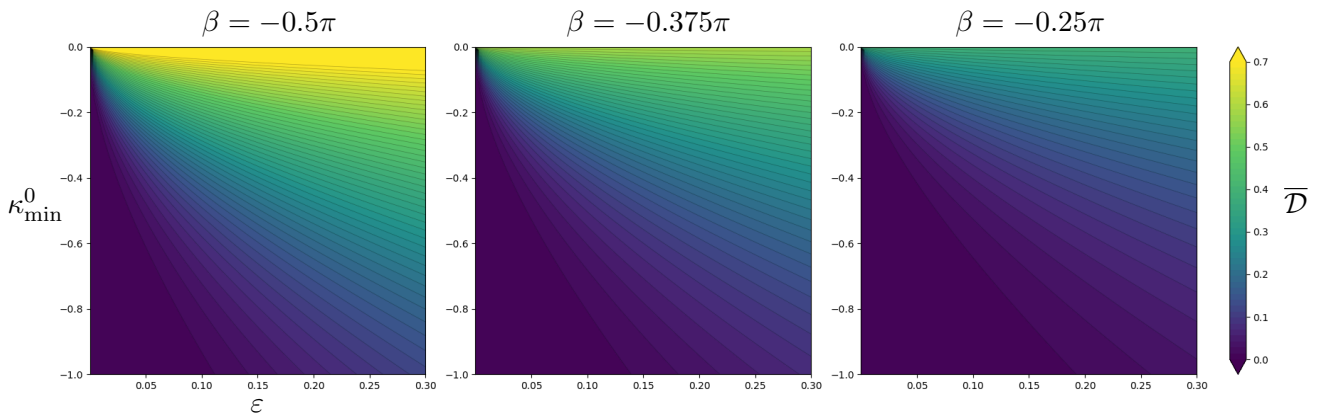
\begin{figure}[h]
    \centering
    \begin{subfigure}[b]{0.45\textwidth}
        \centering
        \begin{overpic}[width=1.3\textwidth]{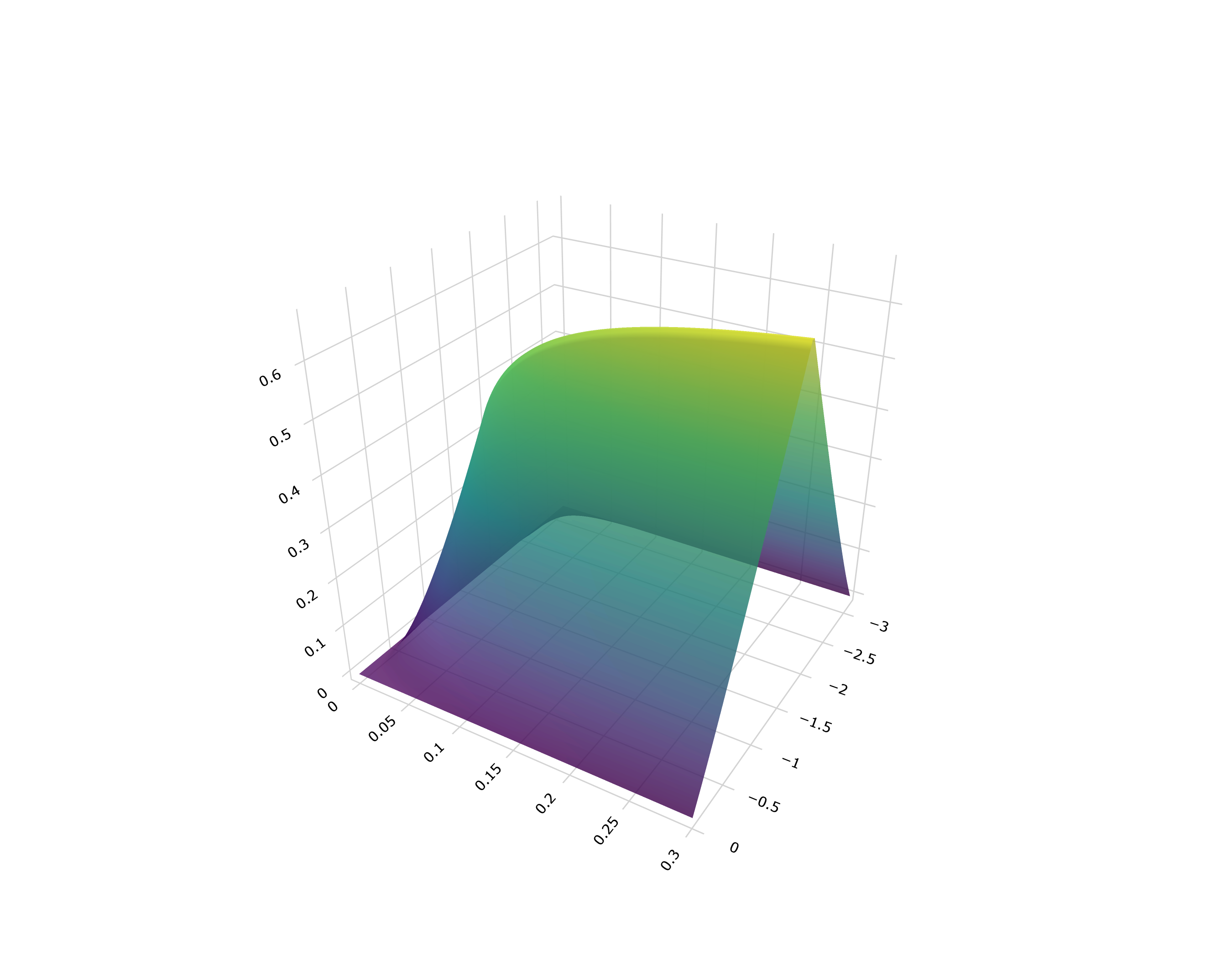}
            \put(15,34){\small$\overline{\mathcal D}$}
            \put(37,8){\small$\varepsilon$}
            \put(69,15){\small$\beta$}
        \end{overpic}

        \vspace{-.5cm}
        
        \caption{$\kappa_{\min}^0=-0.1$}
    \end{subfigure}
    \hspace{-1cm}
    \begin{subfigure}[b]{0.45\textwidth}
        \centering
        \begin{overpic}[width=1.3\textwidth]{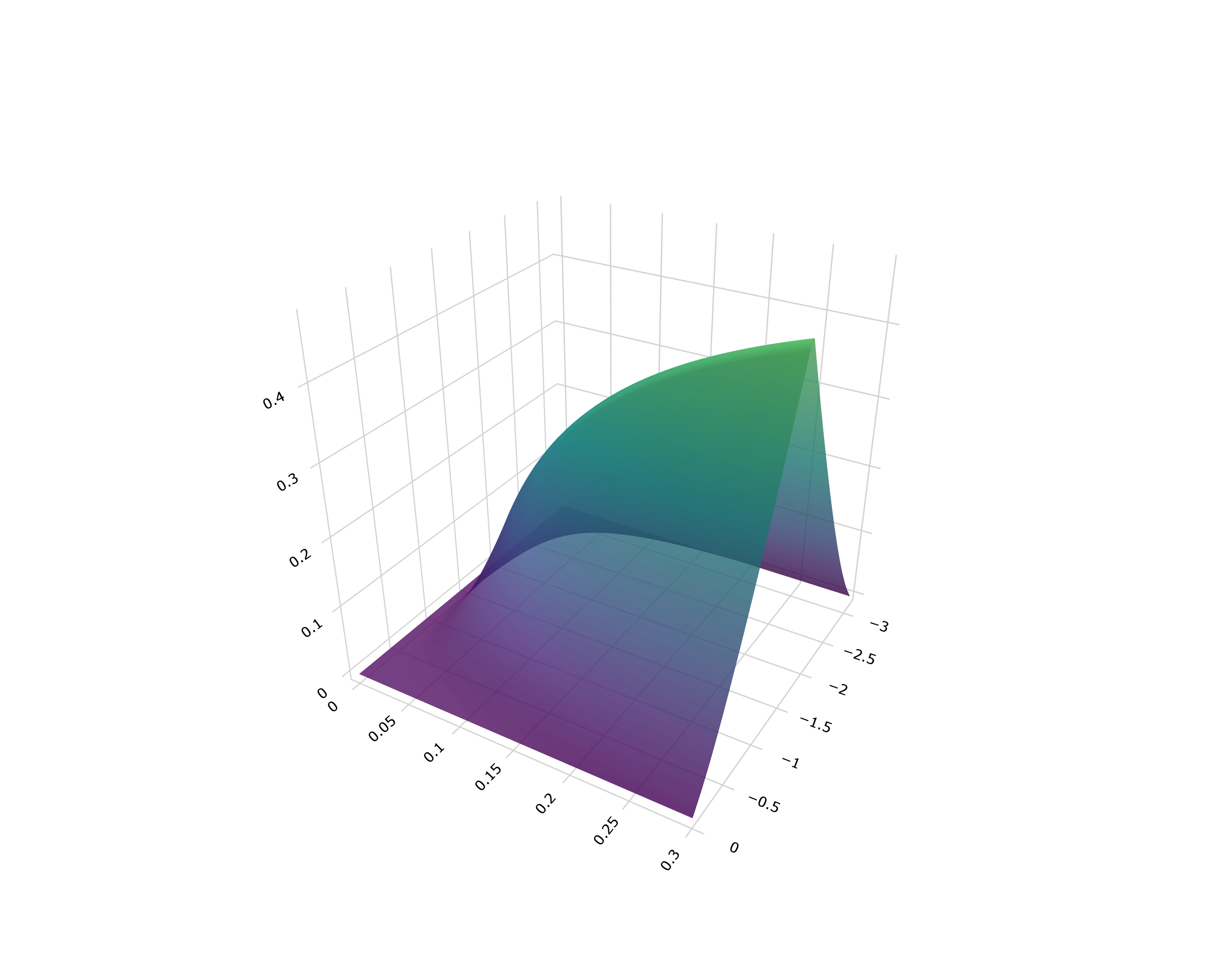}
            
        \end{overpic}

        \vspace{-.5cm}
        
        \caption{$\kappa_{\min}^0=-0.3$}
        % \label{sub_}
    \end{subfigure}

    \vspace{-1cm}
    
    \begin{subfigure}[b]{0.45\textwidth}
        \centering
        \begin{overpic}[width=1.2\textwidth]{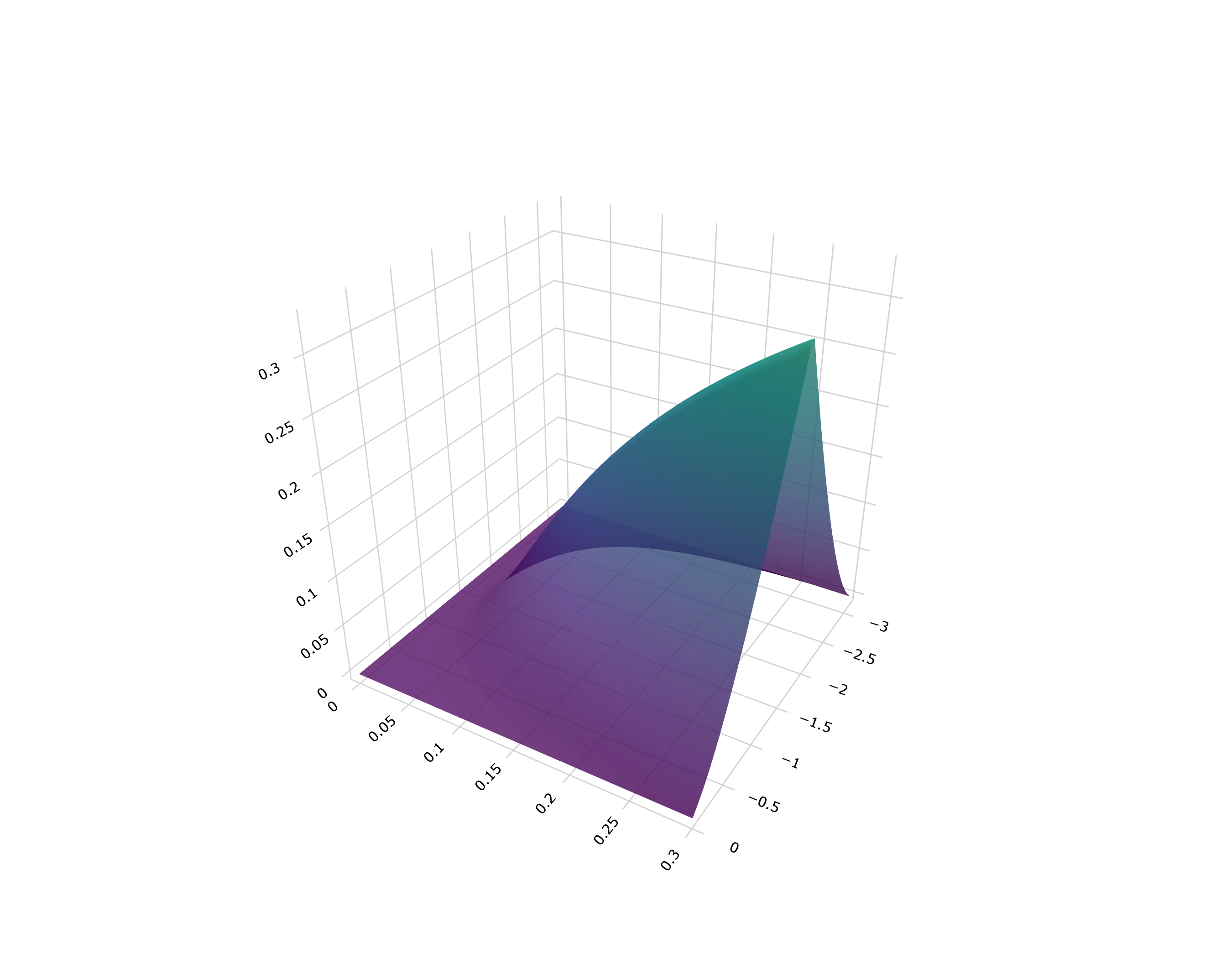}
            
        \end{overpic}

        \vspace{-.5cm}
        
        \caption{$\kappa_{\min}^0=-0.5$}
    \end{subfigure}
    \hspace{-1cm}
    \begin{subfigure}[b]{0.45\textwidth}
        \centering
        \begin{overpic}[width=1.2\textwidth]{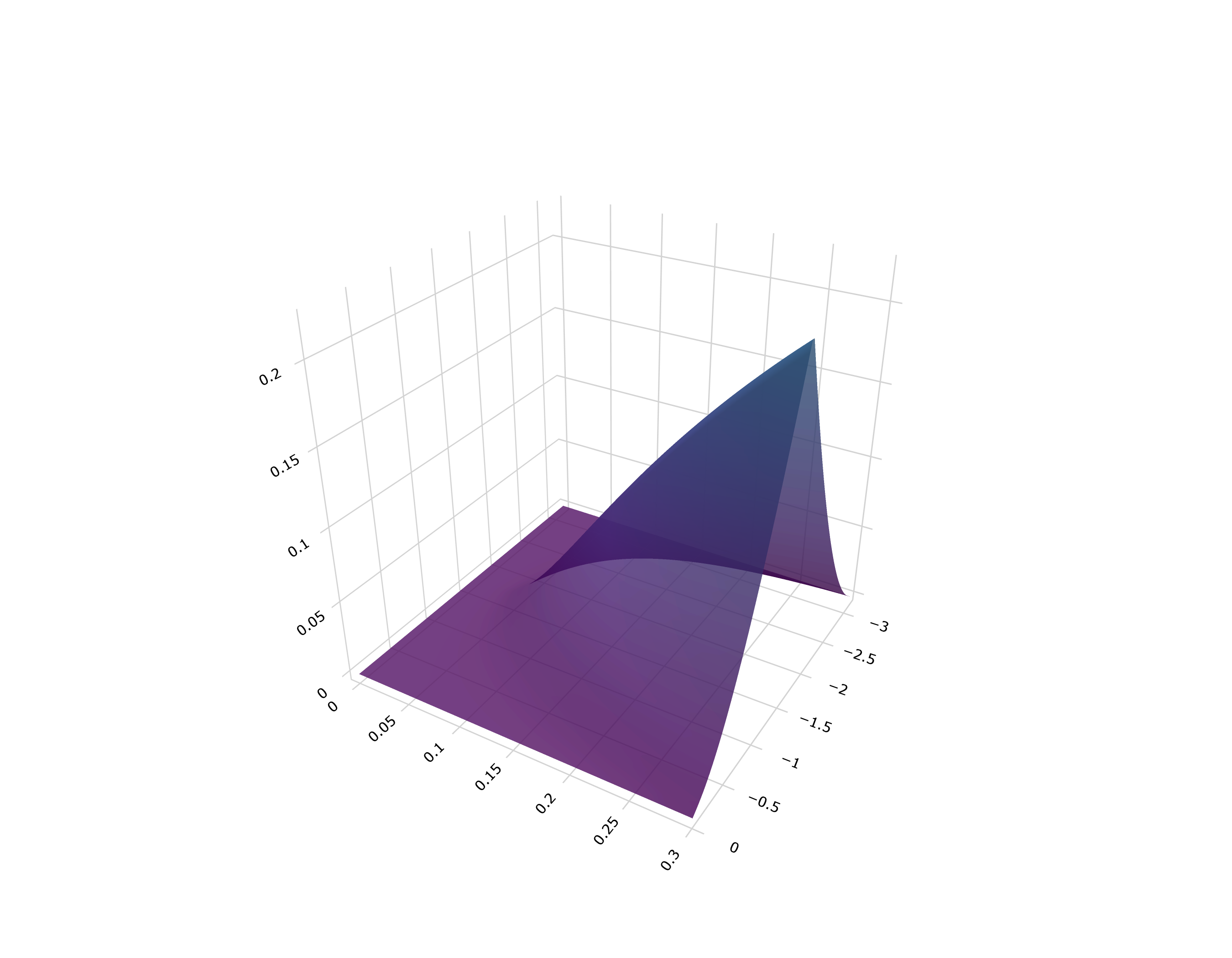}
            
        \end{overpic}

        \vspace{-.5cm}
        
        \caption{$\kappa_{\min}^0=-0.7$}
    \end{subfigure}
    \caption{Three-dimensional plots of the critical diameter $\overline{\mathcal D}$ as a function of the phase-lag parameter $\beta$ and the adaptation rate $\varepsilon$ for several fixed values of the minimal initial coupling $\kappa_{\min}^0$. The surface height is color-coded according to the same scale as in \Cref{fig_cond_2d}.}
    \label{fig_cond_3d}
\end{figure}

\section{Conclusion}\label{sec_conclusion}
In this work, we investigated Kuramoto dynamics on two canonical classes of signed networks, block-structured networks and locally excitatory–globally inhibitory (LEGI) networks, with a particular emphasis on the local stability of phase-locked equilibria. By explicitly characterizing the spectra of the corresponding Jacobian matrices, we identified structural constraints that limit the stability of clustered and antipodal phase-locked states in static signed networks.

These observations naturally lead to the consideration of coupling adaptation. For the adaptive Kuramoto model, we derived explicit sufficient conditions on the initial phase configuration, coupling strengths, and system parameters under which complete synchronization or antipodal synchronization emerges. Our analysis quantitatively clarifies how the adaptation rate $\varepsilon$ enlarges the basin of attraction and enables stable antipodal states that are not supported by the underlying static network alone.

Several directions remain open for future work. While our discussion has been restricted to equilibria on these two classes of networks, the main perspective advocated here is to move beyond specific network instances and toward a classification based on structural properties of asymptotic networks. This viewpoint suggests promising directions for future research on adaptive Kuramoto dynamics, in particular on how different adaptive rules give rise to distinct asymptotic network structures and associated stable phase-locked states.

\appendix
\section{Proof of \Cref{prop_eigval_antipodal}}\label{sec_appendix}
Recall that the interaction matrix $A$ defined in \eqref{A_def} has a block structure determined by the partition of groups into two phase clusters $\mathcal{G}^{(0)}$ and $\mathcal{G}^{(\pi)}$. Let $N_0 = \sum_{m \in \mathcal{G}^{(0)}} |\mathcal{G}_m|$ and $N_\pi = \sum_{m \in \mathcal{G}^{(\pi)}} |\mathcal{G}_m|$ denote the total number of oscillators in each cluster, with $N = N_0 + N_\pi$.

The eigenvalue problem is given by $L \mathbf{v} = \lambda \mathbf{v}$. We decompose the eigenspace into two orthogonal subspaces: the subspace of vectors summing to zero within each group, and the subspace of group-wise constant vectors.

\subsection*{A.1. Local modes (Intra-group dynamics)}
Consider an eigenvector $\mathbf{v}=(v_1,\cdots,v_N)$ that is non-zero only on the nodes belonging to a single group $\mathcal{G}_m$ and satisfies $\sum_{i \in \mathcal{G}_m} v_i = 0$. For any node $i \in \mathcal{G}_m$, the entry $(A\mathbf{v})_i$ involves only interactions within $\mathcal{G}_m$, since $v_j = 0$ for $j \notin \mathcal{G}_m$. Within the group, the coupling strength is $a$. Thus,
\begin{align*}
    (A\mathbf{v})_i = a \sum_{j \in \mathcal{G}_m} v_j = 0.
\end{align*}
Since $A\mathbf{v} = \mathbf{0}$, the eigenvalue equation simplifies to $L \mathbf{v} = D_A \mathbf{v}$. Therefore, $\mathbf{v}$ is an eigenvector with eigenvalue $\lambda = (D_A)_{ii}$.
The row sum $(D_A)_{ii}$ for $i \in \mathcal{G}_m$ is calculated as:
\begin{itemize}
    \item If $m \in \mathcal{G}^{(0)}$:
    \begin{align*}
        (D_A)_{ii} &= a|\mathcal{G}_m| + b(N_0 - |\mathcal{G}_m|) - b N_\pi \\
        &= (a-b)|\mathcal{G}_m| + b(N_0 - N_\pi).
    \end{align*}
    \item If $m \in \mathcal{G}^{(\pi)}$:
    \begin{align*}
        (D_A)_{ii} &= a|\mathcal{G}_m| + b(N_\pi - |\mathcal{G}_m|) - b N_0 \\
        &= (a-b)|\mathcal{G}_m| - b(N_0 - N_\pi).
    \end{align*}
\end{itemize}
These two cases can be combined into the single expression found in \Cref{prop_eigval_antipodal}:
$$
\lambda = (a-b)|\mathcal{G}_m| + b\left( \sum_{k \in \mathcal{G}^{(0)}} |\mathcal{G}_k| - \sum_{k \in \mathcal{G}^{(\pi)}} |\mathcal{G}_k| \right) (2\mathbf{1}_{\mathcal{G}^{(0)}}(m) - 1).
$$
For each group $\mathcal{G}_m$, the constraint $\sum v_i = 0$ imposes one linear restriction, yielding multiplicity $|\mathcal{G}_m| - 1$. Summing over all $m$, these local modes account for $\sum_{m=1}^M (|\mathcal{G}_m| - 1) = N - M$ eigenvalues.

\subsection*{A.2. Global modes (Inter-group dynamics)}
The remaining $M$ eigenvalues correspond to eigenvectors $\mathbf{v}$ that are constant within each group. Let $\mathbf{v}$ be defined by the vector $\mathbf{u} = (u_1, \dots, u_M)^\top$, where $v_i = u_m$ for all $i \in \mathcal{G}_m$.

The dynamics of these group averages allow for specific structured solutions:

\paragraph{1. Synchronization mode ($\lambda = 0$)}
The vector $\mathbf{v} = \mathbf{1}$ (all $u_m = 1$) is clearly in the kernel of the Laplacian $L$, corresponding to the simplified eigenvalue $\lambda = 0$.

\paragraph{2. Antipodal mode ($\lambda = -bN$)}
Consider a vector defined by the phase clusters: $u_m = u^{(0)}$ for $m \in \mathcal{G}^{(0)}$ and $u_m = u^{(\pi)}$ for $m \in \mathcal{G}^{(\pi)}$, satisfying the weighted orthogonality condition $N_0 u^{(0)} + N_\pi u^{(\pi)} = 0$.
The row sums of $A$ act on this vector as follows. For $i \in \mathcal{G}_m \subset \mathcal{G}^{(0)}$:
$$
(A\mathbf{v})_i = a |\mathcal G_m| u^{(0)} + b(N_0 - |\mathcal G_m|) u^{(0)} - b N_\pi u^{(\pi)}.
$$
Substituting $u^{(\pi)} = -(N_0/N_\pi)u^{(0)}$ yields:
$$
(A\mathbf{v})_i = \Big((a-b)|\mathcal G_m| + b N_0 + b N_0\Big)u^{(0)} =\Big((a-b)|\mathcal G_m| + b(N+N_\pi)\Big) u^{(0)}.
$$
Calculating $(L\mathbf{v})_i = (D_A)_{ii} u^{(0)} - (A\mathbf{v})_i$, and simplifying, we obtain $\lambda = -bN$. This is a simple eigenvalue associated with the global phase difference between the two antipodal clusters.

\paragraph{3. Cluster splay modes}
The remaining eigenvalues arise from deviations within the clusters $\mathcal{G}^{(0)}$ and $\mathcal{G}^{(\pi)}$ that preserve the cluster means.
Consider a vector where $u_m$ varies for $m \in \mathcal{G}^{(\pi)}$ such that $\sum_{m \in \mathcal{G}^{(\pi)}} |\mathcal{G}_m| u_m = 0$, and $u_m = 0$ for all $m \in \mathcal{G}^{(0)}$.
For any $k \in \mathcal{G}^{(\pi)}$, the eigenvalue equation reads:
\[\lambda u_k = (D_A)_{kk} u_k - \left( a|\mathcal{G}_k|u_k + \sum_{l \in \mathcal{G}^{(\pi)}, l \neq k} b|\mathcal{G}_l| u_l \right).\]
Since $\sum_{l \in \mathcal{G}^{(\pi)}, l \neq k} |\mathcal{G}_l| u_l = -|\mathcal{G}_k| u_k$ and
\[(D_A)_{kk}=(a-b)|\mathcal G_k|+b(N_\pi-N_0),\]
the interaction term simplifies, leading to:
$$
\lambda = b(N_\pi - N_0) = b\left( 2 \sum_{m \in \mathcal{G}^{(\pi)}} |\mathcal{G}_m| - N \right).
$$
The number of linearly independent vectors of this type is $|\{m : m \in \mathcal{G}^{(\pi)}\}| - 1$.
Similarly, variations within $\mathcal{G}^{(0)}$ yield eigenvalues $\lambda = b(N_0 - N_\pi)$ with multiplicity $|\{m : m \in \mathcal{G}^{(0)}\}| - 1$.

Combining these results yields the spectrum characterized in \Cref{prop_eigval_antipodal}. \qed


\begin{thebibliography}{10}
	
	\bibitem{AS2004}
	Daniel~M. Abrams and Steven~H. Strogatz.
	\newblock Chimera states for coupled oscillators.
	\newblock {\em Phys. Rev. Lett.}, 93:174102, 2004.
	
	\bibitem{Altafini2013}
	Claudio Altafini.
	\newblock Consensus problems on networks with antagonistic interactions.
	\newblock {\em IEEE Transactions on Automatic Control}, 58(4):935--946, 2013.
	
	\bibitem{BGKKY2023}
	Rico Berner, Thilo Gross, Christian Kuehn, Jürgen Kurths, and Serhiy Yanchuk.
	\newblock Adaptive dynamical networks.
	\newblock {\em Physics Reports}, 1031:1--59, 2023.
	
	\bibitem{BSY2019}
	Rico Berner, Eckehard Sch\"{o}ll, and Serhiy Yanchuk.
	\newblock Multiclusters in networks of adaptively coupled phase oscillators.
	\newblock {\em SIAM Journal on Applied Dynamical Systems}, 18(4):2227--2266,
	2019.
	
	\bibitem{Cartwright1956}
	Dorwin Cartwright and Frank Harary.
	\newblock Structural balance: a generalization of heider's theory.
	\newblock {\em Psychological review}, 63(5):277, 1956.
	
	\bibitem{CMM2018}
	Hayato Chiba, Georgi~S. Medvedev, and Matthew~S. Mizuhara.
	\newblock Bifurcations in the kuramoto model on graphs.
	\newblock {\em Chaos: An Interdisciplinary Journal of Nonlinear Science},
	28(7):073109, 2018.
	
	\bibitem{DJD2019}
	Robin Delabays, Philippe Jacquod, and Florian D\"{o}rfler.
	\newblock The kuramoto model on oriented and signed graphs.
	\newblock {\em SIAM Journal on Applied Dynamical Systems}, 18(1):458--480,
	2019.
	
	\bibitem{Gallier2016}
	Jean Gallier.
	\newblock Spectral theory of unsigned and signed graphs. applications to graph
	clustering: a survey.
	\newblock {\em arXiv preprint arXiv:1601.04692}, 2016.
	
	\bibitem{GB2008}
	Thilo Gross and Bernd Blasius.
	\newblock Adaptive coevolutionary networks: a review.
	\newblock {\em Journal of the Royal Society Interface}, 5(20):259--271, 2008.
	
	\bibitem{HNP2016}
	Seung-Yeal Ha, Se~Eun Noh, and Jinyeong Park.
	\newblock Synchronization of kuramoto oscillators with adaptive couplings.
	\newblock {\em SIAM Journal on Applied Dynamical Systems}, 15(1):162--194,
	2016.
	
	\bibitem{Haugland2021}
	Sindre~W Haugland.
	\newblock The changing notion of chimera states, a critical review.
	\newblock {\em Journal of Physics: Complexity}, 2(3):032001, 2021.
	
	\bibitem{HS2011}
	Hyunsuk Hong and Steven~H. Strogatz.
	\newblock Kuramoto model of coupled oscillators with positive and negative
	coupling parameters: An example of conformist and contrarian oscillators.
	\newblock {\em Phys. Rev. Lett.}, 106:054102, 2011.
	
	\bibitem{Jaros2018}
	Patrycja Jaros, Serhiy Brezetsky, Roman Levchenko, Dawid Dudkowski, Tomasz
	Kapitaniak, and Yuri Maistrenko.
	\newblock Solitary states for coupled oscillators with inertia.
	\newblock {\em Chaos: An Interdisciplinary Journal of Nonlinear Science},
	28(1):011103, 2018.
	
	\bibitem{KYSN2017}
	D.~V. Kasatkin, S.~Yanchuk, E.~Sch\"oll, and V.~I. Nekorkin.
	\newblock Self-organized emergence of multilayer structure and chimera states
	in dynamical networks with adaptive couplings.
	\newblock {\em Phys. Rev. E}, 96:062211, 2017.
	
	\bibitem{Menck2013}
	Peter~J Menck, Jobst Heitzig, Norbert Marwan, and J{\"u}rgen Kurths.
	\newblock How basin stability complements the linear-stability paradigm.
	\newblock {\em Nature physics}, 9(2):89--92, 2013.
	
	\bibitem{RZ2007}
	Quansheng Ren and Jianye Zhao.
	\newblock Adaptive coupling and enhanced synchronization in coupled phase
	oscillators.
	\newblock {\em Phys. Rev. E}, 76:016207, 2007.
	
	\bibitem{Sawicki2023}
	Jakub Sawicki, Rico Berner, Sarah~AM Loos, Mehrnaz Anvari, Rolf Bader, Wolfram
	Barfuss, Nicola Botta, Nuria Brede, Igor Franovi{\'c}, Daniel~J Gauthier,
	et~al.
	\newblock Perspectives on adaptive dynamical systems.
	\newblock {\em Chaos: An Interdisciplinary Journal of Nonlinear Science},
	33(7), 2023.
	
	\bibitem{SYT2002}
	Philip Seliger, Stephen~C. Young, and Lev~S. Tsimring.
	\newblock Plasticity and learning in a network of coupled phase oscillators.
	\newblock {\em Phys. Rev. E}, 65:041906, 2002.
	
	\bibitem{SAB2019}
	Guodong Shi, Claudio Altafini, and John~S. Baras.
	\newblock Dynamics over signed networks.
	\newblock {\em SIAM Review}, 61(2):229--257, 2019.
	
	\bibitem{WSG2006}
	Daniel~A. Wiley, Steven~H. Strogatz, and Michelle Girvan.
	\newblock The size of the sync basin.
	\newblock {\em Chaos: An Interdisciplinary Journal of Nonlinear Science},
	16(1):015103, 2006.
	
	\bibitem{ZC2014}
	Hongwei Zhang and Jie Chen.
	\newblock Bipartite consensus of general linear multi-agent systems.
	\newblock {\em 2014 American Control Conference}, pages 808--812, 2014.
	
\end{thebibliography}
\end{document}